\documentclass{amsart}

\usepackage{nicefrac}
\usepackage{slashbox}
\usepackage{float}

\usepackage{amssymb}
\usepackage{graphicx} 

\usepackage{subfig}
\usepackage{comment}

\renewcommand{\Pr}{\mathrm{Pr}}

\newtheorem{theorem}{Theorem}[section]
\newtheorem{lemma}[theorem]{Lemma}
\newtheorem{cor}[theorem]{Corollary}

\theoremstyle{definition}
\newtheorem{definition}[theorem]{Definition}

\theoremstyle{remark}

\numberwithin{equation}{section}

\newcommand{\Ref}[1]{(\ref{#1})}
\newcommand{\nn}{\nonumber \\}
\newcommand{\ds}{\displaystyle}

\newcommand{\abar}{a^{-1}}
\newcommand{\bbar}{b^{-1}}

\newcommand{\dee}[1]{\mathrm{d}#1}
\newcommand{\diff}[2]{\frac{\dee{#1}}{\dee{#2}}}

\def\W#1{\widetilde{#1}}

\newcommand{\LS}{MR0577064}

\newcommand{\AGLP}{MR2473819}

\newcommand{\Cohen}{Cohen}
\newcommand{\DykemaBS}{DykemaBS}

\newcommand{\BartV}{MR2176547}
\newcommand{\Grig}{Grig}
\newcommand{\Kouksov}{MR1689726}
\newcommand{\ERW}{MR3043436}

\newcommand{\KouksovRational}{MR1487319}

\newcommand{\BurClearyWiest}{MR2395786}

\newcommand{\Mann}{MR2894945}
\newcommand{\Wagon}{MR1251963}

\newcommand{\MadrasSlade}{MR1197356}

\newcommand{\BuksMC}{janse2009monte}
\newcommand{\Metropolis}{metropolis}

\newcommand{\TesiMonte}{tesimonte}
\newcommand{\Geyer}{geyer}

\newcommand{\BB}{MR2197808}
\newcommand{\TatchMoore}{MR3095713}
\newcommand{\DykemaRW}{MR2216708}
\newcommand{\BartVirag}{MR2176547}
\newcommand{\BartKaimNekrash}{MR2730578}
\newcommand{\BartWoess}{MR2131635}
\newcommand{\Lalley}{MR2275700}

\newcommand{\NagA}{MR2087798}

\newcommand{\NagB}{MR1691645}
\newcommand{\DiacSCa}{MR1650316}
\newcommand{\DiacSCb}{MR1414925}
\newcommand{\DiacSCd}{MR1254308}
\newcommand{\DiacSCe}{MR1245303}
\newcommand{\DiacSCf}{MR1233621}

\newcommand{\OrtnerWoess}{MR2338235}
\newcommand{\Woess}{MR731608}
\newcommand{\BScogrowth}{BScogrowth}

\begin{document}

\title[Random sampling of trivials words]{Random sampling of trivials words\\ in
finitely presented groups}

\author[M. Elder]{M. Elder}
\address{School of Mathematical \& Physical Sciences, The~University~of~Newcastle, Callaghan, New South Wales, Australia}
\email{murray.elder@newcastle.edu.au}
\author[A. Rechnitzer]{A. Rechnitzer}
\address{Department of Mathematics, University of British Columbia, Vancouver, British Columbia, Canada}
\email{andrewr@math.ubc.ca}
\author[E.~J.~Janse~van~Rensburg]{E.~J.~Janse~van~Rensburg}
\address{York University, Toronto, Ontario, Canada}
\email{rensburg@mathstat.yorku.ca}



\date{\today}

\begin{abstract}
We describe a novel algorithm for random sampling of freely reduced words equal 
to the identity in a finitely presented group. The algorithm is 
based on Metropolis Monte Carlo sampling. The algorithm samples from a 
stretched Boltzmann distribution
\begin{align*}
  \pi(w) &= (|w|+1)^{\alpha} \beta^{|w|} \cdot Z^{-1}
\end{align*}
where $|w|$ is the length of a word $w$, $\alpha$ and $\beta$ are parameters 
of the algorithm, and $Z$ is a normalising constant. It follows that words of 
the same length are sampled with the same probability. 
The distribution can be expressed in terms of the cogrowth series of the 
group, which then allows us to  relate statistical properties of words sampled by 
the algorithm to the cogrowth of the group, and hence its amenability.

We have implemented the algorithm and applied it to several  group 
presentations including the Baumslag-Solitar groups, some free 
products studied by Kouksov, a finitely presented  amenable  
group that is not subexponentially amenable (based on the 
{\em basilica group}), and  Richard Thompson's group~$F$.
\end{abstract}

\maketitle

\section{Introduction}
In this article we propose a new random sampling algorithm for finitely 
presented groups. The algorithm samples freely reduced words in the generators 
that are equal to the identity of the group. This algorithm is based on ideas 
from statistical mechanics and Markov chain theory. In particular, the 
algorithm is inspired by the BFACF algorithm for sampling self-avoiding 
polygons (we refer the reader to \cite{\BuksMC, \MadrasSlade} 
for a description of BFACF and self-avoiding polygons). The algorithm differs 
from previous work on random walks in groups in that it only samples trivial 
words. Indeed, it can be seen as executing a random walk on the space of trivial 
words, rather than a random walk on the Cayley graph of the group.

We prove that the algorithm coverges to a specified distribution, and relate this distribution to the cogrowth series of 
the group. 
By varying a parameter, we can detect numerically the precise position of the radius of converge of the cogrowth series, and 
hence numerically predict the amenability or non-amenability of the group.

We have implemented the algorithm and have applied it to a selection of 
finitely presentated groups. These include several Baumslag-Solitar groups, some free products whose cogrowth series were studied by 
Kouksov \cite{\Kouksov}, a finitely 
presented relative of the basilica group,  and R. Thompson's group $F$. 

The present article continues previous work by the authors \cite{\BScogrowth, \ERW}, where various techniques, also based in statistical mechanics 
and enumerative combinatorics, were applied to the problem of estimating and 
computing the cogrowth of groups. This in turn built on previous work of 
Burillo, Cleary and Wiest \cite{\BurClearyWiest}, and  Arzhantseva, Guba,  
Lustig, and Pr\'eaux \cite{\AGLP}, who applied experimental techniques to the 
problem of deciding the amenability of Thompson's group~$F$. In other work, 
Belk and Brown \cite{\BB} proved the currently best known upper bound for the 
isoperimetric constant for $F$, and  Moore \cite{\TatchMoore} gives lower 
bounds on the growth rate of F\o lner function for $F$. 

More generally a (by no means exhaustive) list of others working in the area of 
random walks on groups is Bartholdi \cite{\BartKaimNekrash, \BartVirag, 
\BartWoess},  Diaconis and Saloff-Coste \cite{\DiacSCd,\DiacSCb,\DiacSCa,\DiacSCe,\DiacSCf}, Dykema 
\cite{\DykemaRW,\DykemaBS}, Lalley \cite{\Lalley}, Smirnova-Nagnibeda  
\cite{\NagB, \NagA} and Woess \cite{\OrtnerWoess,\Woess}.

For the benefit of readers outside of group theory, and to establish notation,  we start with a precise
definition of group presentations and cogrowth.

\begin{definition}[Presentations and trivial words]
A presentation 
\begin{align}
\langle a_1,\dots,a_k \mid  R_1, \dots, R_\ell, \dots \rangle
\end{align}
 encodes a (finitely generated) group as follows.
\begin{itemize}

\item
Let  $\mathcal{S} = \{a_1, a_1^{-1}, \dots,a_k, a_k^{-1} \}$ be a set of $2k$ distinct letters, and
 $\mathcal S^*$ the set of all finite strings or {\em words} over the letters in $\mathcal S$. 
 
\item A word in $\mathcal{S}^*$ is called {\em freely reduced} if it contains 
no subword of the form $a_ia_i^{-1}$ or $a_i^{-1}a_i$  for any $a_i\in \mathcal S$.

\item The set of all freely reduced words, together with the operation of
concatenation followed by free reduction (deleting $a_ia_i^{-1}$ and $a_i^{-1}a_i$ pairs) forms a 
group, called the {\em free group} on the letters $\{a_1,\dots,a_k\}$, which we 
denote by  $F(a_1,\ldots, a_k).$

\item Let $R_1, \dots, R_\ell, \dots$ be a finite or infinite list of distinct words in $F(a_1,\ldots, a_k).$

\item Let $N(R_1, \dots, R_\ell,\dots)$ be the normal subgroup of the free group  consisting of all words
of the form $\ds \prod_{j=1}^m \rho_jR_j\rho_j^{-1}$ after free reduction, where $\rho_i$ is
any element in the free group, and $R_j$ is one of the relators or their inverses.
This subgroup is called the {\em normal closure} of the set of relators.

\item The group encoded by the  presentation
$\langle a_1,\dots, a_k \mid  R_1, 
\dots, R_\ell, \dots \rangle$
 is defined to be the quotient group 
 $F(a_1,\ldots, a_k)/N(R_1, \dots, R_\ell,\dots)$. 

\item  The letters $a_i$ are  called {\em
generators}, and the words $R_i$  are called {\em 
relations} or {\em relators}.

\item  A group $G$ is called {\em finitely generated} if it can be encoded by a 
presentation with the list $a_1,\ldots, a_k$ finite, and {\em finitely 
presented} if it can be encoded by a  presentation with both  lists $a_1,\ldots, 
a_k$ and $R_1, \dots,R_\ell$ finite. In this article the list $a_1, \dots , a_k$ will always be finite.

\item It follows that a word in $F(a_1,\ldots, a_k)$ equals the identity element
in $G$  if and only if it lies in the normal subgroup $N(R_1, \dots, R_\ell,\dots)$, 
and so is equal to a product of conjugates of relators and their inverses.

\end{itemize}
\end{definition}
We will make extensive use of this last point in the work below. We call a word
in $F(a_1,\ldots, a_k)$ that equals the identity element in $G$ a {\em trivial
word}.

Let $c(n)$ be the number of freely reduced words, $w \in \mathcal{S}^*$, of 
length $n$ that represent the identity of a finitely generated group. This 
function is called the {\em cogrowth function} and the corresponding generating 
function is called the \emph{cogrowth series}. The rate of exponential growth 
of the cogrowth function is the {\em cogrowth} of the group (with respect to a 
chosen finite  generating set). Equivalently the cogrowth is the reciprocal of 
the radius of convergence of the cogrowth series. Grigorchuk and
independently Cohen \cite{\Cohen, \Grig} proved that a finitely generated group
is  amenable if and only if its cogrowth is $|\mathcal{S}|-1$.

For more background on amenability and  cogrowth see \cite{\Mann, \Wagon}.
The free group on two (or more) letters, as defined above, is
known to be non-amenable. Also, subgroups of amenable groups are also amenable.
It follows that if a group contains a subgroup isomorphic to the free
group on 2 generators ($F(a_1,a_2)$ above), then it cannot be amenable.

It is important to note that in some cases the letters in $\mathcal S$ may represent the same group element, 
for example,  consider the presentation $\langle a \mid a^2\rangle$, where the relation $a^2$ implies that  $a=a^{-1}$.
In this example $|\mathcal{S}|=2$ (the letters $a, a^{-1}$ are considered distinct formal symbols), and the cogrowth 
function is $c(0)=1, c(2n)=2, c(2n+1)=0$. The cogrowth series is  then
\begin{align}
\sum c(n)z^n= 1+2z^2+2z^4+\dots  
  =  \frac{1+z^2}{1-z^2},
\end{align} and one can see directly that the radius of convergence is $1=|\mathcal S|-1$.
Note that Kouksov  \cite{\KouksovRational} showed that a group has rational cogrowth series if and only if it is finite.

The article is organised as follows. In Section~\ref{sec alg} we describe the 
algorithm for sampling trivial words from a given finite presentation. We then 
analyse the algorithm and show that it samples from a stretched Boltzmann 
distribution (Corollary~\ref{cor main}). In Section~\ref{sec results} we apply 
the algorithm to several finite presentations. In cases where the cogrowth 
series is known, we see excellent agreement between the exact results and 
numerical data generated by our algorithm (for both amenable and non-amenable 
groups). We also apply the algorithm to sample words from groups for 
which the cogrowth series is not known, including Thompson's group~$F$. We 
summmarise our results in Section~\ref{sec conc}.

\section{Metropolis Sampling of Freely Reduced Trivial Words in Groups}
\label{sec alg}
Let $G = \langle a_1,\dots a_k | R_1, \dots ,R_\ell \rangle$ be a 
finitely presented group, and let $\mathcal{X}$ be the set of all freely 
reduced  trivial
words in $G$. We assume that the words $R_i$ are freely reduced and non-empty.
Define a set $\mathcal{R}$ as follows. Take all 
the relators $R_i$, their inverses $R_i^{-1}$, and all cyclic permutations 
of these. The set $\mathcal{R}$ consists of all of these words  {after} 
free reduction. For example, in the case of $\mathrm{BS}(2,3) = \langle a,b 
\mid a^2ba^{-3}\bbar \rangle$ the single relator yeilds $2\times 7=14$ elements 
in $\mathcal{R}$.

We will describe an algorithm which samples a sequence of freely reduced 
trivial words 
\begin{align}
  ( w_0,w_1,w_2, \ldots,w_n,\ldots); && w_i \in \mathcal{X}.
\end{align}
We refer to the words $w_i$ as \emph{states}. The 
algorithm constructs a new state $w_{n+1}$ from the current state $w_n$ by 
applying one of two \emph{elementary moves} with specified probabilities that 
depend only on $w_n$. Such a procedure is known as a Markov chain. 

There are two parts to the selection rule in the Markov chain --- the 
\emph{elementary moves} which transform $w_n$ to $w_{n+1}$ and the 
\emph{probabilities} with which they are implemented. The implementation we use 
is known as Metropolis sampling \cite{\Metropolis}. In this way our algorithm is a Metropolis 
algorithm sampling along a Markov chain in $\mathcal{X}$. 

\subsection{Elementary moves}\label{subsec moves}

In this subsection we describe several \emph{elementary moves} that we will perform on words in $\mathcal X$ to obtain other words in $\mathcal X$. 
Our goal is to define a set of moves that have a well defined \emph{reverse move}, and such that any two words in $\mathcal X$ are connected by a finite sequence of moves.

The moves we describe are all based on the following two operations: conjugation by $x\in\mathcal S$; and insertion of $R\in \mathcal R$.  
 For technical reasons which will describe below, we 
consider only what we call \emph{left-insertions} rather than arbitrary insertions of relators.
 The elementary moves are as follows.

On input $w\in\mathcal X$:
\begin{itemize}
  \item \textit{(Conjugation by $x$)} Let $x \in \mathcal{S}$.  Write 
$w^\prime = x w x^{-1}$ and perform free reductions on $w^\prime$ to produce 
$w^{\prime\prime}$. Return $w^{\prime\prime}$.

\item \textit{(Left-insertion of $R$ at position $m$)}  Let $R\in \mathcal{R}$ 
and $m \in\{0,1,\ldots,|w|\}$. Partition $w$ into two subwords $u$ and $v$, 
with $|v| = m$. Form $w^\prime = u R v$, and freely reduce this word by first freely reducing $uR$, 
obtaining $u^\prime v$, and then freely reducing to obtain
$w^{\prime\prime}$. If $m=0$, then $R$ is appended to $w$, and if $m=|w|$, 
then $R$ is prepended to $w$.

Return $w^{\prime\prime}$ unless a symbol of $v$ is cancelled during the free-reduction step (\emph{i.e.} a 
cancellation occurs to the right of $R$).
If this occurs then we set $w'' = w$ and return $w''$ (and so return a copy 
of the original word $w$).

\end{itemize}

Note that conjugations change word length by at most $2$, and left-insertions by at most $|R|$.

Since $|\mathcal S|, |\mathcal R|$ and words $w\in \mathcal X$ are all finite, there are finitely many possible elementary moves from a state  $w$ to a state $u$.
The next two lemmas show that elementary moves are ``uniquely reversible" in the sense that if there are 
$p$ conjugations and $q$ left-insertions from a state $w$ to a state $z$, then the same number of 
each type send $z$ to $w$. 

For example, if $R=abc\in \mathcal R$, $w=abcabcabc$ and $z=abcabcabcabc$,  there are exactly  4 left-insertions of $R$ possible in $w$ to obtain $z$, and exactly 4 left-insertions in $z$ to get $w$.

\begin{lemma}\label{lem conj rev} 
Let $w,z \in\mathcal X$ with $w\neq z$. If $z$ is obtained from $w$ by a conjugation move, then either:
\begin{itemize}\item  there is exactly one conjugation move from $w$ to $z$, and exactly one conjugation move from $z$ to $w$; or
\item  there are exactly two conjugation moves from $w$ to $z$, and exactly two conjugation moves from $z$ to $w$. In this case $w=(xy)^n$ and $z=(yx)^n$ for some $x,y\in \mathcal S$.\end{itemize}
\end{lemma}
\begin{proof}
Suppose $x,y$  are distinct symbols in $\mathcal S$, and $z$ is obtained from $w$ by conjugation by either $x$ or $y$.

\begin{itemize} 
\item If $xwx^{-1}$ is freely reduced, then $ywy^{-1}$ must  freely reduce to $z=xwx^{-1}$, and since both words have the same length, they must be identical and $x,y$ are the same symbol.

 \item If  $w=x^{-1}w_1$
and $w_1x^{-1}$ is freely reduced,  then $yx^{-1}w_1y$ must freely reduce to $z=w_1x^{-1}$, so must contain a cancellation.
If $yx^{-1}$ is a free reduction then $x$ and $y$ are the same symbol, so the cancellation must be in $w_1y$, so $w_1=w_2y^{-1}$.
So $z=w_2yx^{-1}=yx^{-1}w_2$ and the two expressions are identical strings, so $w_2$ must be a product of $(yx^{-1})$ pairs, so $w=(yx^{-1})^n$.
In this case we have exactly two conjugations from $(yx^{-1})^n$ to $(x^{-1}y)^n$, and exactly two back the other way (namely conjugation by $x^{-1}$ or $y^{-1}$).

\item If $w=w_1x$ and $xw_1$ is freely reduced, then $yw_1xy^{-1}$ must freely reduce to $z=xw_1$ so contains a cancellation. Since $x,y$ are assumed distinct the cancellation must be in $yw_1$, so $w_1=y^{-1}w_2$ and $w_2xy^{-1}=xy^{-1}w_2$ are identical strings, so $w=(xy^{-1})^n$ and we have exactly two conjugations from $(xy^{-1})^n$ to $(y^{-1}x)^n$
and back.

\item If $w=x^{-1}w_1x$ then $yx^{-1}w_1xy^{-1}$ must freely reduce to $z=w_1$, so since $x^{-1}w_1x$ is freely reduced (it is in $\mathcal X$) we must have $x,y$ are the same symbol.
\end{itemize}
\end{proof}

\begin{lemma}\label{lem left rev}
Let $w,z \in\mathcal X$ with $w\neq z$. If $w\to z$ by insertion of $R\in\mathcal R$ at position $m$, then $z\to w$ by insertion of $R^{-1}\in\mathcal R$ at position $m$.
\end{lemma}
\begin{proof}
Let $w=uv$ with $|v|=m$. If $uRv$ is not freely reduced then we have $u=u_1u_2, R=u_2^{-1}r$, and $u_1r$ is freely reduced. Then $w=u_1u_2v, z=u_1rv$.
Note that by definition there is no cancellation of the suffix $v$. 

Then left-inserting $R^{-1}$ at position $m$ in $z$ gives $u_1R^{-1}v=u_1rr^{-1}u_2v=u_1u_2v=w$.
\end{proof}

Note that for arbitrary insertions of relators, the previous lemma does not hold. For example consider the group 
$\mathbb{Z}^2=\langle a,b \ | \ ba\bbar\abar\rangle$ and let $w=a^3b^4\abar\bbar aba^{-4}b^{-4}$ and $z=a^4b^4a^{-3}b^{-4}$.
Inserting the relator $R=ba\bbar\abar$ into $w$
at $m=9$ gives
\begin{align}
a^3 b^4 \abar\bbar \cdot R \cdot  ab a^{-3}b^{-4} & \longrightarrow  \   a^3 b^4 \abar\bbar \cdot  ba\bbar\abar  \cdot  ab a^{-3}b^{-4}   \nn
& \longrightarrow \  a^3 b^3 \abar  \cdot  ab a^{-3}b^{-4} \nn
& \longrightarrow \  a^3 b^3 b a^{-3}b^{-4}
\end{align}
This move is not a left-insertion since there is cancellation to the right of the inserted relator.
Suppose it were allowed. Then there is no way to  obtain $w$ via insertion of $R^{-1}=ab\abar\bbar$ at any position in $z$, as one can easily verify by trying each position.
By restricting to only left-insertions we avoid such problems, and guarantee that elementary moves have well defined reverse moves.

\begin{lemma}
\label{lem connected}
  Let $G, \mathcal{S}, \mathcal{R}, \mathcal{X}$ be as above. Let $w \in 
\mathcal{X}$ then there exists a finite sequence of conjugations and 
left-insertions that transform the empty word to $w$.
\end{lemma}
\begin{proof}
A word $u \in\{a_1^{\pm 1}, \ldots, a_k^{\pm 1}\}^*$ represents the identity 
element in $G$ if and only if it is the product of conjugates of the relators 
$R_i^{\pm 1}$.  So since $w \in \mathcal{X}$, it can be written as the product
\begin{align}
 \displaystyle \prod_{j=1}^n \rho_jr_j\rho_j^{-1}
\end{align}
after free reduction, where $\rho_j \in \mathcal{S}^*$ and $r_j=R_{i_j}^{\pm 
1}$.

We can obtain $w$ using conjugation and left-insertion as follows:
\begin{itemize}
\item set $u$ to be the empty word;
\item left-insert $r_1$ after which  $u=r_1$;
\item conjugate by $\rho_2^{-1}\rho_1$ one letter at a time to obtain
$u=\rho_2^{-1}\rho_1r_1\rho_1^{-1}\rho_2$ after free reduction;
\item left-insert $r_2$ at the extreme right ($m=0$);
\item repeat the previous two steps (conjugating by
$\rho_{j+1}^{-1}\rho_j$ then left-inserting $r_j$ at the extreme right) until 
$r_n$ is left-inserted at the extreme right;
\item conjugate by $\rho_n$.
\end{itemize}
Since we only ever append $r_j$ to the extreme right of the word, there are 
no right cancellations. 
\end{proof}
Note that since conjugations and left-insertions are reversible it follows that 
given any two words in $\mathcal{X}$ there is some finite sequence of 
elementary moves that transforms one to the other.

The reader may find it useful to consider the set $\mathcal X$ of states as the vertices of a graph, with states connected by directed edges 
if there is an elementary move from one to another, labeled by $(\mathrm{conj}, x)$  if it is conjugation by $x\in \mathcal S$, 
and $(\mathrm{insert}, R, m)$  if it is a left-insertion of   $R\in\mathcal R$ at position $m\in\mathbb N$.
The above lemmas prove that 
each edge between distinct states has a unique corresponding reverse edge with appropriate label, and 
 that the graph is connected.

\subsection{Transition probabilities}\label{subsec trans}
In this subsection we define  probabilities with which elementary moves are selected or 
rejected.

Let $p_c \in (0,1)$, $\alpha \in \mathbb{R}$ and $\beta \in (0,1)$ 
be parameters of the algorithm. Fix a probability distribution, $P$, over 
$\mathcal{R}$, so that $P(R)$ is the probability of choosing $R\in 
\mathcal{R}$. Further, assume that $P(R)>0$ for all $R\in \mathcal{R}$ and also 
that $P(R) = P(R^{-1})$. Since $\mathcal{R}$ is finite, the obvious choice of $P$ is the uniform distribution 
--- indeed this is what we used in our implementation. The algorithm we 
describe can easily be modified for presentations with infinitely many relators 
by choosing an appropriate distribution on $\mathcal{R}$ in this case --- see 
subsection \ref{subsec inf} below.

Let $w_n$ be the current word. We construct the next word, $w_{n+1}$ as follows:
\begin{itemize}
\item With probability $p_c$ choose to perform a conjugation, otherwise (with 
probability $1-p_c$) perform a left-insertion.
\item If conjugation is selected,  choose $x \in \mathcal{S}$  with 
uniform probability and perform a conjugation  by $x$ as described above to obtain $w^{\prime\prime}$. 
Then $w_{n+1}$ is chosen according to the rule
\begin{align}
 w_{n+1} &= \begin{cases}
    w^{\prime\prime}, & \mbox{with probability } \min\left\{ 
1, 
\frac{(|w^{\prime\prime}|+1)^{1+\alpha}}{(|w|+1)^{1+\alpha}}
\cdot 
\frac{\beta^{|w^{\prime\prime}|}}{\beta^{|w|}} 
\right\}; \\
    w_n, & \mbox{ otherwise} .
             \end{cases}
\label{eqn2}
\end{align}
\item If left-insertion is selected,  choose $R\in \mathcal{R}$ with 
probability $P(R)$ and a location $m\in \{0,1,2,\ldots,|w_n|\}$ with uniform 
probability. Peform a left-insertion of $R$ at $m$ as described above to 
obtain $w^{\prime\prime}$. Then $w_{n+1}$ is chosen according to the rule
\begin{align}
 w_{n+1} &= \begin{cases}
    w^{\prime\prime}, & \mbox{with probability } \min\left\{ 
1, 
\frac{(|w^{\prime\prime}|+1)^{\alpha}}{(|w|+1)^{\alpha}}
\cdot 
\frac{\beta^{|w^{\prime\prime}|}}{\beta^{|w|}} 
\right\}; \\
    w_n, & \mbox{ otherwise} .
             \end{cases}
\label{eqn3} 
\end{align}
\end{itemize}

An implementation of a Markov chain which includes probabilistic rules under which moves are accepted or 
rejected is known as a \emph{Metropolis style} algorithm. By including these specific rejection probabilities, 
we are able to establish the  \emph{detailed balance} condition, which we describe next.
Notice that equations~\Ref{eqn2} and~\Ref{eqn3} are very similar except that the 
power of $1+\alpha$ is changed to $\alpha$. This small difference is 
required in order to satisfy the {detailed balance} condition.

We point out to the reader that the Markov chain we have described is 
\emph{not} a random walk on the Cayley graph of the group. Rather it executes a 
random walk on the set of trivial words $\mathcal{X}$. We can think of two 
points $x_1,x_2 \in \mathcal{X}$ being connected by a weighted directed edge if 
the corresponding words are linked by a single conjugation or left-insertion 
where the weight is the appropriate probability.

\subsection{The sample distribution} 
In this subsection we prove properties of the Markov chain defined by the 
transitions described above. Much of the results in this section are standard in the theory of 
Markov chains, but for completeness we include all relevant details.
First let us define some useful notation. Define
\begin{align}
  \Pr(u \to v) &= \text{probability of tranforming $u$ to $v$ by an 
elementary move.}
\end{align}
as per equations~\Ref{eqn2} and~\Ref{eqn3}. Define
\begin{align}
  \Pr_n(u \to v) &= \text{probability of tranforming $u$ to $v$ by $n$ 
elementary moves.}
\end{align}

\begin{definition}
 A Markov chain is said to be \emph{irreducible} if  there is a non-zero 
probability of moving between any two given states in a finite number of 
elementary moves.

 A Markov chain is said to be \emph{aperiodic} when for any two states $u,v$  
there exists an integer $N_0$ so that for all $N>N_0$
\begin{align*}
  \Pr_N(u \to v) >0.
\end{align*}
That is, if the algorithm is in state $u \in \mathcal{X}$, there 
is a positive probability of reaching $v \in \mathcal{X}$ in $N$ elementary 
moves for all $N>N_0$.

A Markov chain that is both irreducible and aperiodic is said to be 
\emph{ergodic}.
\end{definition}

\begin{lemma}
 The Markov chain is irreducible.
\end{lemma}
\begin{proof}
By Lemma~\ref{lem connected}, there exists a sequence of moves that transforms 
any given state to any other given state. The probability of executing that 
sequence is positive, since the probability of any one move in the sequence is 
positive.
\end{proof}

\begin{lemma}
\label{lem aperiod}
 The Markov chain is aperiodic.
\end{lemma}
\begin{proof}
By the Lemma~\ref{lem connected}, for any $u,v \in \mathcal{X}$ there is finite 
sequence of elementary moves that starts at $u$ and finishes at $v$. Let $N_0$ 
be the length of this sequence. Once the chain reaches this final state, $v$, 
there is a positive probability that any further moves leaves the algorithm 
in the same state. Thus the algorithm is aperiodic.
\end{proof}

The previous two lemmas imply that the the Markov is ergodic since
it is both irreducible and aperiodic.

\begin{definition}
  Let $\pi$ be some probability distribution over the state space of 
a given Markov chain. The chain is said to satisfy the \emph{detailed 
balance} condition with respect to $\pi$ when 
\begin{align*}
  \pi(u) \cdot \Pr(u \to v)
  &= \pi(v) \cdot \Pr(v \to u)
\end{align*}
for any two states $u,v$ in the chain.
\end{definition}
Note that $\pi$ is a probability distribution over the states, while $\Pr(u \to 
v)$ is the probability of a particular transition in the Markov chain. Detailed 
balance describes how these probabilities interact. The main reason to consider 
detailed balance is that it implies that $\pi$ is the \emph{stationary 
distribution} under the Markov chain, which we now define.

\begin{definition}
A probability distribution $\pi$ over the states of a Markov chain 
is \emph{stationary} if
\begin{align*}
  \pi(u) &= \sum_v \Pr(v \to u) \pi(v)
\end{align*}
That is, $\pi$ is unchanged by a single step of the chain.
\end{definition}
\begin{lemma}\label{lem det balance}
 If a Markov 
chain satisfies detailed balance with respect to $\pi$, then $\pi$ is 
stationary.
\end{lemma}
\begin{proof}
 Assume that detailed balance is satisfied, then 
\begin{align}
  \pi(u) \cdot \Pr(u \to v)
  &= \pi(v) \cdot \Pr(v \to u).
\end{align}
Summing over all states $v$ then gives
\begin{align}
  \pi(u) \sum_v \Pr(u \to v) &= \sum_v \pi(v) \Pr(v \to u)
\end{align}
Since $\sum_v \Pr(u \to v) =1$ the result follows.
\end{proof}

\begin{lemma}\label{lem detailed balance}
  Let $\pi$ be a probability distribution on $\mathcal{X}$ given by
  \begin{align*}
    \pi(u) &= \frac{(|u|+1)^{1+\alpha} \beta^{|u|} }{Z}
  \end{align*}
  where $Z$ is a normalising constant. The Markov chain defined above 
satisfies the detailed balance condition with respect to $\pi$.
\end{lemma}
We note that the normalising constant exists and is finite when $\beta$ is 
sufficiently small. We discuss this point further in the next section.

\begin{proof}
Let $u,v \in \mathcal{X}$. There are three possibilities: there is no single 
elementary move transforming $u$ to $v$ or vice-versa; $u$ and $v$ are 
separated by a single conjugation move; $u$ and $v$ are separated by a single 
left-insertion. 

If there is no single elementary move between $u$ and $v$, then 
$\Pr(u \to v)=\Pr(v \to u) = 0$ and the detailed balance condition is trivially 
satisfied.

Now suppose that $v$ was obtained from $u$ by a conjugation as described
above. Define 
\begin{align}
p_{uv} &= \frac{(|v|+1)^{1+\alpha}}{(|u|+1)^{1+\alpha}}
\frac{\beta^{|v|}}{\beta^{|u|}}
& p_{vu} &= p_{uv}^{-1}.
\end{align}
The transition probabilities are 
\begin{align}
  \Pr (u\to v) &= \frac{1}{|\mathcal{S}|} \min \{1,p_{uv}\}, &
  \Pr (v\to u) &= \frac{1}{|\mathcal{S}|} \min \{1,p_{vu}\}.
\end{align}
The factor of $|\mathcal{S}|$ arises because we have to choose the correct 
conjugating element from $\mathcal{S}$. Note that $p_{uv} \leq 1$ if and only if 
$p_{vu} \geq 1$. So without loss of generality, assume that $p_{uv} \leq 1, 
p_{vu} \geq 1$. Then
\begin{align}
\Pr (u \to v) &= \frac{p_{uv} }{|\mathcal{S}|},
& \Pr (v \to u) &= \frac{ 1 }{|\mathcal{S}|}.
\end{align}
Hence we have
\begin{align}
\Pr (u\to v) &= p_{uv} \cdot \Pr (v\to u).
\label{eqn puv conj}
\end{align}

Next assume that $v$ is obtained from $u$ by a left-insertion of $R \in 
\mathcal{R}$. Let
\begin{align}
q_{uv} &= \frac{(|v|+1)^{\alpha}}{(|u|+1)^{\alpha}}
\frac{\beta^{|v|}}{\beta^{|u|}} &
q_{vu} &= q_{uv}^{-1}
\end{align}

The transition probabilities are given by
\begin{align}
\Pr(u \to v) &= \frac{P(R)}{|w|+1} \min \{1, q_{uv} \},& 
\Pr(v \to u) &= \frac{P(R^{-1})}{|v|+1} \min \{1, q_{vu} \}
\end{align}
where $P(R)$ is the probability of choosing the relation $R$ and the factor of 
$|u|+1$ arises from choosing the correct position to insert $R$. Recall that 
$P$ was chosen so that $P(R) = P(R^{-1})$ for any $R \in \mathcal{R}$. Without 
loss of generality assume that $q_{uv} \leq 1$ so that $q_{vu} \geq 1$ and then
\begin{align}
\Pr(u\to v) &= \frac{P(R)}{|u|+1} \cdot q_{uv}, &
\Pr(v\to u) &= \frac{P(R^{-1})}{|v|+1} = \frac{P(R)}{|v|+1}
\end{align}
and so
\begin{align}
\Pr(u\to v) 
&= \frac{|v|+1}{|u|+1} \cdot q_{uv} \cdot \Pr(v\to u) 
=  p_{uv} \cdot \Pr(v\to u) 
\label{eqn puv li}
\end{align}

Notice equation~\Ref{eqn puv conj} is identical to equation~\Ref{eqn puv li}. 
This equation can be rewritten in a more symmetric form as
\begin{align}
(|u|+1)^{1+\alpha} \beta^{|u|} \Pr (u\to v) &=
(|v|+1)^{1+\alpha} \beta^{ |v|} \Pr (v\to u).
\label{eqn6} 
\end{align}
Dividing by the normalising constant we obtain the detailed balance criterion
\begin{align}
  \pi(u) \cdot \Pr(u \to v)
  &= \pi(v) \cdot \Pr(v \to u).
\end{align}

As noted above, it is possible that two states $u,v$ are connected by more than one elementary move. In this case   $\Pr (u\to v)$  is the sum of the probabilities for 
each elementary move, as is $\Pr (v\to u)$ (by Lemmas~\ref{lem conj rev} and \ref{lem left rev}), so  detailed balance is preserved.
\end{proof}

The next result shows that detailed balance implies uniqueness of the 
stationary distribution. Though it is a standard result in the theory of Markov 
chains, again we include it here for completeness.
\begin{lemma}
The distribution $\pi$ described in the previous lemma is the unique 
distribution on $\mathcal{X}$ for which the algorithm satisfies detailed 
balance.
\end{lemma}
\begin{proof}
 Suppose there is another distribution $\varphi$ on $\mathcal{X}$ for 
which detailed balance is satisfied. If $\varphi \neq \pi$ then there exists a 
state $y \in \mathcal{X}$ so that $\varphi(y) > \pi(y)$. 

So for every state $y'$ that is connected to $y$ by an elementary move 
(\emph{i.e.} for which $\Pr(y \to y')>0$) we have
\begin{align}
  \varphi(y') \Pr(y' \to y)  &= 
  \varphi(y) \Pr(y \to y') > \pi(y) \Pr(y \to y') = \pi(y') \Pr(y' \to y)
\end{align}
Hence $\varphi(y')>\pi(y')$. Thus $\varphi(x)>\pi(x)$ for all $x$ reachable from 
$y$. 
Since the chain is irreducible, $\varphi(x) > \pi(x)$ for all $x \in 
\mathcal{X}$. This contradicts the assumption that $\varphi$ is 
a probability distribution.
\end{proof}

Now that we have established the above properties of the Markov chain, we can 
make use of the Fundamental Theorem of Markov chains:
\begin{theorem}[Fundamental Theorem of Markov chains]
If a Markov chain $\mathcal{M}$ is irreducible and aperiodic then it has a 
unique stationary distribution~$\varphi$. Moreover,
\begin{align*}
  \Pr_n(x \to y) &\to \phi(y) & \text{as } n \to \infty
\end{align*}
for all $x,y$ in the state space of $\mathcal{M}$.
\end{theorem}
The above theorem can be found in most standard 
 texts on stochastic 
processes --- see, for example, \cite{karlin1975, mitzenmacher2005, 
rosenthal2006}.

\begin{cor}
\label{cor main}
  Given any two states $u,v \in \mathcal{X}$
  \begin{align*}
    \Pr_n(u \to v) & \to \pi(v) & \text{as } n \to \infty
  \end{align*}
  where $\pi(u)$ is the unique stationary distribution of the Markov chain
  \begin{align*}
    \pi(u) &= \frac{(|u|+1)^{1+\alpha} \beta^{|u|} }{Z}
  \end{align*}
  where $Z$ is a normalising constant which depends on $\alpha, \beta, 
p_c$ and the group presentation.
\end{cor}
\begin{proof}
By the previous lemmas, our Markov chain satisfies the conditions of the 
theorem. Further, since $\pi$ is a stationary distribution for our 
Markov chain, it must, by the same theorem, be the unique stationary 
distribution.
\end{proof}

The above corollary implies that we can use our Markov chain to sample 
trivial words from a given finitely presented group with a specific 
distribution, $\pi$. When $\alpha=-1$, $\pi$ is the {\em Boltzmann} or {\em Gibbs 
distribution}. For other values of $\alpha$ we can think of $\pi$ as a 
``stretched'' Boltzmann distribution. Also note that $\pi$ does not depend of 
the details of the word, but only on its length. So if two words have the same 
length then they are sampled with the same probability. 

In the next section we examine the mean length of sampled words and describe 
how this can inform us about the cogrowth of the group.

\subsection{Mean length of sampled words}
As demonstrated in the previous section, the Markov chain converges to a 
stretched Boltzmann distribution, $\pi$. We defined $\pi$ above in terms of a 
normalising constant, $Z$, which we now make more precise. Since we require 
$\sum_{w\in \mathcal{X}} \pi(w) = 1$, we must have
\begin{align}
  Z &= \sum_{w \in \mathcal{X}} (|w|+1)^{1+\alpha} \beta^{|w|} 
  \intertext{which can be written in terms of the cogrowth function}
  Z &= \sum_{n  \geq 0 } c(n) \cdot (n+1)^{1+\alpha} \beta^{n}.
\end{align}
This sum converges to a finite value for $0 \leq \beta < \beta_c$, where 
$\beta_c$ is 
\begin{align}
  \beta_c &= \limsup_{n\to\infty} c(n)^{-1/n}.
\end{align}
and is independent of the parameters $\alpha, p_c$. Note that, $\beta_c$ is 
exactly the radius of convergence of cogrowth series
\begin{align}
  C(z) &= \sum_{n\geq0} c(n) z^n.
\end{align}
This demonstrates the link between the behaviour of the Markov chain and the 
cogrowth of the underlying group.

Let us now turn to expected length of words sampled by the Markov chain. Under 
the stationary distribution, $\pi$, the expected length of words in 
$\mathcal{X}$ is given by
\begin{align}
  \mathbb{E}( |w| )
&= \sum_{w \in \mathcal{X}} |w| \pi(w) 
= \sum_w |w| \frac{ (|w|+1)^{1+\alpha} \beta^{|w|} }{Z}.\\
&= \frac{\sum_{n\geq0} n (n+1)^{1+\alpha} c(n) \beta^{n} }{Z}.
\end{align}

With the Markov chain as described we can select a particular value of $\beta$ 
and compare samples from the chain to exact results for groups where the 
cogrowth series is known (such as $\mathbb{Z}^2$). In practice we would like to 
examine how the expected length changes with $\beta$.
When $\beta$ is very small, $\mathbb{E}(|w|) \equiv \langle n \rangle$ should 
be small since shorter words are favoured. As $\beta$ grows the expectation 
will 
increase. When $\beta$ exceeds $\beta_c$ the expectation will cease to converge 
and words sampled by the chain will become longer and longer.

Rather than running many independent copies of the chain at distinct 
$\beta$-values we use a technique known as \emph{Multiple Markov chains} or
\emph{parallel tempering} which samples at a set of distinct $\beta$-values 
simultaneously. We refer the reader to \cite{\Geyer, \TesiMonte} for a detailed 
description of this method.

When $\alpha=-1$ we can write the mean length explicitly as the log-derivative 
of~$C(z)$:
\begin{align}
  \mathbb{E}( |w| )
&= \left.\left(\frac{z C'(z)}{C(z)} 
\right)\right|_{z=\beta}
  = \left.\left( z \diff{}{z} \log C(z) \right)\right|_{z=\beta}.
\end{align}
One can do similarly for $\alpha=0,1,2,\dots$
\begin{align}
\label{eqn exp alpha}
  \mathbb{E}( |w| )
&= \left.\left( 
\underbrace{
\diff{}{z} z 
\diff{}{z} z
\cdots
}_{1+\alpha \text{ times}}
  \left(z \diff{}{z} \log 
C(z)\right) \right)\right|_{z=\beta}.
\end{align}
We will make use of this expression in Section~\ref{sec results} for groups where the 
cogrowth series is known exactly. This will allow us to compare numerical 
results from an implementation of the Markov chain against exact results. Note 
that in the graphs that follow below we will use $\langle n \rangle$ to denote 
mean length in place of $\mathbb{E}(|w|)$.

\subsection{Alternate sets of elementary moves}
While we have implemented the above Markov chain using conjugations and 
left-insertions as elementary moves, other moves are possible. The proof of 
Lemma~\ref{lem connected} relies on conjugations but only a 
subset of left-insertions. In particular, it only requires left-insertions in $w$ at position $m=0$, that is,  {\em appending} a relation to the extreme right of $w$.

 Hence Corollary~\ref{cor 
main} would still hold for a Markov chain using the following elementary moves
\begin{itemize}
 \item conjugation by $x\in\mathcal{S}$: given $w$, $w' = x^{-1} w x$, 
and
 \item append $R\in\mathcal{R}$: given $w$, $w' = w \mapsto w R$.
\end{itemize}
Note that appending $R$ is always reversible by appending $R^{-1}$. 

Since every word in the state space of the chain represents the identity 
element of the group, we could also introduce a rotation move
\begin{itemize}
 \item rotate at $k$: given $w = uv$ with $|u|=k$, $w'= vu$.
\end{itemize}
In order to ensure this move is reversible by another rotation one needs to 
ensure that no cancellations occur upon freely reducing $w'$. With this 
restriction a rotation by $k$ can always be reversed by a rotation at $|w|-k$.

Of course, if the set of elementary moves is changed then the transition 
probabilities described by equations~\Ref{eqn puv conj} and~\Ref{eqn puv li} 
need to be updated in order to satisfy detailed balance.

\subsection{Avoiding the empty word}
The Markov chain can be implemented to sample from the state space of non-empty 
trivial words. ie from $\mathcal{X}' = \mathcal{X}-\{ \epsilon \}$. To 
do this we alter equations~\Ref{eqn puv conj} and~\Ref{eqn puv li} so that if 
$w'' = \epsilon$ then $w_{n+1} = w_n$. That is, if an elementary move attempts 
to step to the empty word then it is rejected and the current word is kept.

The following lemma shows that with this restriction the Markov chain remains
irreducible.
\begin{lemma}
The Markov chain described above with elementary moves altered to avoid the 
empty word is irreducible on $\mathcal{X}'$, except when applied to the 
presentation $\langle a \mid a^k \rangle, k \in \mathbb{N}$.
\end{lemma}
\begin{proof}
Let $w$ be a word in $\mathcal{X}'$. By Lemma~\ref{lem connected} it can be 
reduced to the empty word by a sequence of elementary words. The penultimate 
word in this sequence must be a relator; denote it $r_w$. 

Hence if $u,v \in \mathcal{X}'$ they can be reduced by sequences of elementary 
moves to relators $r_u, r_v$. There are three possibilities
\begin{itemize}
 \item If $r_u = r_v$ then reversing one of the sequences of moves shows that 
$u,v$ are connected a sequence of elementary moves.

\item If $r_u \neq r_v, r_v^{-1}$ then write $r_u = r_1 r_2$ and 
$r_v = r_2^{-1} r_3$, where $r_2$ is as large as possible. Note that $r_2$ 
could be the empty word. Now
\begin{itemize}
 \item left-insert $r_v$ after $r_u$ to obtain $r_u r_v = r_1 r_2 r_2^{-1} 
r_3 \mapsto r_1 r_3 \neq \epsilon$ (after free reduction).
 \item left-insert $r_1^{-1} r_2^{-1}$ after $r_1$ to obtain $r_1 
r_1^{-1} r_2^{-1} r_3 \mapsto r_2^{-1} r_3 = r_v$.
\end{itemize}
Note that since $r_u = r_1 r_2 \in \mathcal{R}$ so is $r_1^{-1} r_2^{-1}$ since 
it is a cyclic rotation of $r_u^{-1} = r_2^{-1} r_1^{-1}$. Thus there must be a 
sequence of elementary moves connecting $u$ and $v$.
\item If $r_u = r_v^{-1}$ then find another relator $r_w \in \mathcal{R}$ so 
that $r_w \neq r_u, r_u^{-1}$ (we discuss the existence of $r_w$ at 
the end of the proof). Now use the previous case to transform $r_u \mapsto 
r_w$ and again to transform $r_w \mapsto r_v$. This creates a sequence of 
elementary moves connecting $u$ and $v$.
\end{itemize}
In all three cases there is a sequence of moves connecting $u$ and $v$. Since 
the probability of each move in the sequence is positive, so the probability of 
the sequence is positive.

Note that the last case breaks down if we are unable to find $r_w \neq r_u, 
r_u^{-1}$. If the presentation has two or more relations, then simply pick 
$r_w$ to be a cyclic permutation of a relation that is not $r_u$. If the group 
has a single relation , then let $r_w$ be a cyclic permutation of 
$r_u$ different from $r_u, r_u^{-1}$. If no such $r_w$ exists then all cyclic 
permutations of $r_u, r_u^{-1}$ must be equal to either $r_u$ or $r_u^{-1}$. We 
now show that this implies the single relation must be of the form $a^k$.

Let $r_u = x_1 x_2 \dots x_k$ be a word in $\mathcal{S}^*$ and let $w = x_2
\dots x_k x_1$ be a cyclic rotation of $r_u$. If $w = r_u$ then we have 
$x_2=x_1, x_3=x_2, \dots, x_k=x_1$ and thus all the symbols in $r_u$ 
must be the same. On the other hand, if $w = r_u^{-1} = x_k^{-1}\dots x_1^{-1}$ 
then we must have that $x_1 = x_1^{-1}$ which is a contradiction. So we must 
have that $r_u = x^k$ for some $x \in \mathcal{S}$.

Now if the group has two or more generators then we can proceed as follows:
\begin{itemize}
 \item Without loss of generality, write $r_u = a^k$. Conjugate by another 
generator $b$ (again without loss of generality) to obtain $b^{-1} a^k b$. 
Left-insert $a^{-k}$ at the end of the word, giving $b^{-1}a^k b a^{-k}$. 
Rotate the word by a sequence of conjugations to $b a^{-k} b^{-1} a^k$. 
Left-insert $a^{-k}$ at the end of the word giving $b a^{-k} b^{-1}$. Finally 
conjugate by $b$ to arrive at $a^{-k} = r_u^{-1} = r_v$.
\end{itemize}
Thus one can connect $r_u$ and $r_v$, and so $u$ and $v$, by a sequence of 
elementary moves. Again, since each move in the sequence has positive 
probability, so does the whole sequence.

Finally if the group has only a single generator then it must be of the form 
$\langle a \mid a^k\rangle$ for some $k \in \mathbb{N}$. Write $w = a^{nk}$ for 
some $n\in\mathbb{N}$. Conjugating $w$ leaves it unchanged, while a 
left-insertion maps $w \mapsto a^{(n-1)k}, a^{nk}, a^{(n+1)k}$. Hence it is not 
possible to transform $a^k$ to $a^{-k}$ by a sequence of elementary moves 
without passing through the empty word $a^0$.
\end{proof}

Notice that the proofs of Lemmas~\ref{lem aperiod} and~\ref{lem detailed 
balance} remain unchanged. Hence Corollary~\ref{cor main} holds and the Markov 
chain on $\mathcal{X}'$ converges to the same stationary distribution. The only 
difference is that the normalising constant, $Z$, changes; it is reduced by 
exactly 1.

When we implemented the Markov chain on $\mathcal{X}$ as described in subsections~\ref{subsec moves} and \ref{subsec trans}, we 
found that it would spend a very large time sampling the empty word. It 
must do this since the empty word is highly probable under the limit 
distribution. In order to force the chain to sample longer words we implemented 
the chain on $\mathcal{X}'$ and used it to generate the results discussed in 
Section~\ref{sec results}. Note that when computing the exact expected mean 
length of the chain on $\mathcal{X}'$ using equation~\Ref{eqn exp alpha}, we 
must ensure that $C(z)$ does not count the empty word and so we replace $C(z)$ 
by $C(z)-1$.

\subsection{Infinitely related groups}\label{subsec inf}
Another possible extension of the algorithm is to consider groups which have 
a finite number of generators but an infinite number of relations, for example
\begin{align}
\mathbb{Z} \wr \mathbb{Z} &= \left\langle a,b \,\middle|\,
\left[a^{b^i}, a^{b^j} \right] \text{ with } i,j \in \mathbb{Z}
\right\rangle.
\end{align}

When performing a left-insertion we choose a particular relation $R \in 
\mathcal{R}$ with probability $P(R)$. The only restrictions on this 
distribution $P$, are that $P(R)>0$ for all $R \in \mathcal{R}$ and that 
$P(R)=P(R^{-1})$. As long as these conditions are satisfied, then 
the detailed balance condition will be satisfied and Corollary~\ref{cor main} 
will hold. Consequently there is no requirement in the above analysis that 
$\mathcal{R}$ be finite. 

We implemented our chain on the above presentation of $\mathbb{Z} 
\wr\mathbb{Z}$ with different choices of $P$. The statistics 
collected from those Markov chains did appear to be independent of the 
distribution $P$, as one would hope, and it was also consistent with the 
amenability of the group. We have not included this extension 
in the present work; we plan to include it in a future work on precise analysis 
of statistics collected from the chain, together with other infinitely presented groups.

\section{Numerical results}
\label{sec results}
In this section we discuss the application of the Markov chain to concrete 
examples of finitely presented groups. We chose a range of amenable and 
non-amenable groups including those for which the cogrowth series is known 
exactly. Additionally we have applied the Markov chain to Thompson's group 
$F$ --- whose amenability is currently an open problem.

The chain on $\mathcal{X}'$ was implemented in \texttt{c++} with words stored as 
linked lists. The linked-list data structure makes the computer code associated 
with conjugation and left-insertion relatively straight-forward. To ensure 
correctness of the implementation, two separate programs were created 
independently by the second and third authors, and results compared. We used the 
GNU Scientific Library\footnote{Available at 
\texttt{http://www.gnu.org/software/gsl/} at time of writing.} to generate 
pseudo-random numbers to decide transitions within the chain. At each beta value 
we sampled approximately $10^{10}$ elementary moves. Each run consisted of 
$100$ $\beta$-values and took approximately 1 week on a single node of a 
computing cluster at the Western Canada Research Grid (Westgrid). Each node was roughly 
equivalent to a modest desktop computer running Linux. 

We remark that for some groups it is easier to compute the generating function 
of the number of \emph{all} words equivalent to the identity, not just those 
that are freely reduced. This series for $\mathbb{Z}^2$, for example, is
\begin{align}
\label{eqn z2 returns}
  D(z) &= \sum_{n \geq 0} \binom{2n}{n}^2 z^{2n}
  = \frac{2}{\pi} K(4z)
\end{align}
where $K(z)$ is the complete elliptic integral of the first kind. We refer the 
reader to \cite{\BScogrowth} for a short proof of the above. It is 
then straight-forward to transform this series to cogrowth series using the 
following result of Woess:
\begin{lemma}[Lemma 1 of \cite{\Woess}]
\label{lemma woe}
  Let $d(n)$ be the number of words of length $n$ equal to the identity in a 
given group presentation, and let $D(z)=\sum d(n) z^n$ be the associated 
generating function. Let $2q = |\mathcal{S}|$, then 
\begin{align}
\label{eqn woe}
C(z) &= \frac{1-q+q\sqrt{1-4(2q-1)z^2}}{1-4q^2z^2} D\left( 
\frac{1-\sqrt{1-4(2q-1)z^2}}{2(2q-1)z} 
\right)
\end{align}
\end{lemma}
We remind the reader that we have implemented the chain on $\mathcal{X}' = 
\mathcal{X}-\{\epsilon\}$ and so we must replace $C(z)$ by $C(z)-1$ when 
computing exact expectations using equation~\Ref{eqn exp alpha}.

\subsection{Application to $\mathbb{Z}^2$}
Our first example is $\mathbb{Z}^2$ which is amenable and whose cogrowth series 
is known exactly (see above). In Figure~\ref{fig bs11} we plot the exact 
expectation of the length of words as a function of $\beta$ with $\alpha=1$ (the solid curve in the plot); 
this curve was computed by combining equations~\Ref{eqn woe},~\Ref{eqn 
z2 returns} and~\Ref{eqn exp alpha}.

We overlay the average length of words observed in the implementation of 
our chain running with $\alpha=1$ and a range of $\beta$-values (the crosses in the plot). The figure 
demonstrates that there is excellent agreement between the exact results and 
the numerical estimates. Similar agreement was found for different values of 
$\alpha$.

\begin{figure}[h!]
 \begin{center}
  \includegraphics[height=5cm]{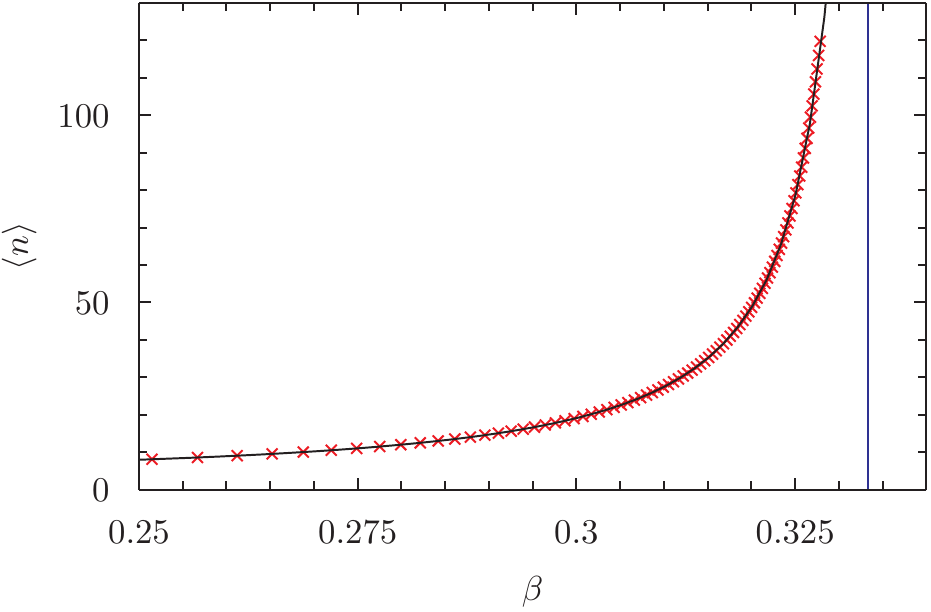}
  \end{center}
\caption{The mean length of sampled words plotted against $\beta$ for 
the standard presentation of $\mathbb{Z}^2$ with $\alpha=1$. The crosses 
indicate data obtained from an implementation of the algorithm while the curve 
indicates the expectation derived from the exact cogrowth series for the group. 
The vertical line indicates $\beta_c=\nicefrac{1}{3}$.}
\label{fig bs11}
\end{figure}

Notice that the data show that the mean length increases with $\beta$ and that 
it becomes larger and larger as $\beta$ approaches $\nicefrac{1}{3}$. Indeed, 
since the group is amenable, we know that the cogrowth is exactly $3$ (by 
\cite{\Grig, \Cohen}), and so the radius of convergence of the cogrowth series 
is $\nicefrac{1}{3}$.

\subsection{Application to examples of Kouksov}
The cogrowth series is known in closed form for very few groups. In 
\cite{\Kouksov} Kouksov gives explicit formulae for some free products. We 
examined the following three:
\begin{align}\label{presKouksov}
 K_1 &= \langle a,b \,|\, a^2, b^3 \rangle, \nn
K_2 & =  \langle a,b \,|\, a^3, b^3 \rangle, \nn
K_3 & =  \langle a,b,c \,|\, a^2, b^2, c^2 \rangle.
\end{align}
whose cogrowth series are given by
\begin{align}
C_1(t) &= \frac{(1+t)\left( f_1(t) + (2-t+6t^2)\sqrt{
f_2(t) }\right)}
{2(1-3t)(1+3t^2)(1+3t+3t^2)(1-t+3t^2)}\\
%
%
C_2(t) &= \frac{(1+t)(-t+\sqrt{1-2t-t^2-6t^3+9t^4})}{(1-3t)(1+2t+3t^2)}, 
\text{ and}\\
%
C_3(t) &= \frac{-1-5t^2+3\sqrt{1-22t^2+25t^4}}{2(1-25t^2)}
\end{align}
where $f_1(t) = -t + t^2-8t^3+3t^4-9t^5$ and $f_2(t) = 
1-2t + t^2-6t^3 - 8t^4-18t^5+9t^6-54t^7+81t^8$. The radii of convergence 
of these cogrowth series are $0.3418821478, 0.3664068598$ 
and $0.2192752634$ respectively (to ten significant digits). Hence the cogrowth 
is strictly smaller than the value required for amenability being 3, 3 and 5, 
respectively. Indeed each of these contains a non-abelian free subgroup and so 
are non-amenable; in the case of the groups $K_1$ and $K_2$ the free subgroups 
are $F((ab),(ab^{-1}))$, and for $K_3$ the free subgroup is $F((ab),(ac))$.

\begin{figure}
  \centering
  \subfloat[$\langle a,b|a^2,b^3 \rangle$ sampled with $\alpha = 0$.]{
    \includegraphics[height=5cm]{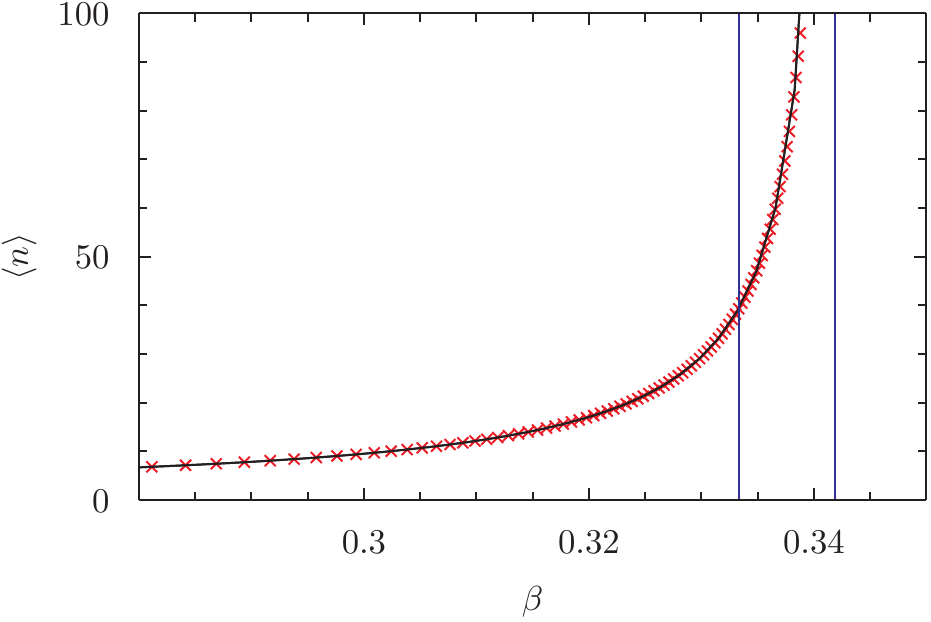}}\\
  \subfloat[$\langle a,b|a^3,b^3 \rangle$ sampled with $\alpha = 0$.]{
    \includegraphics[height=5cm]{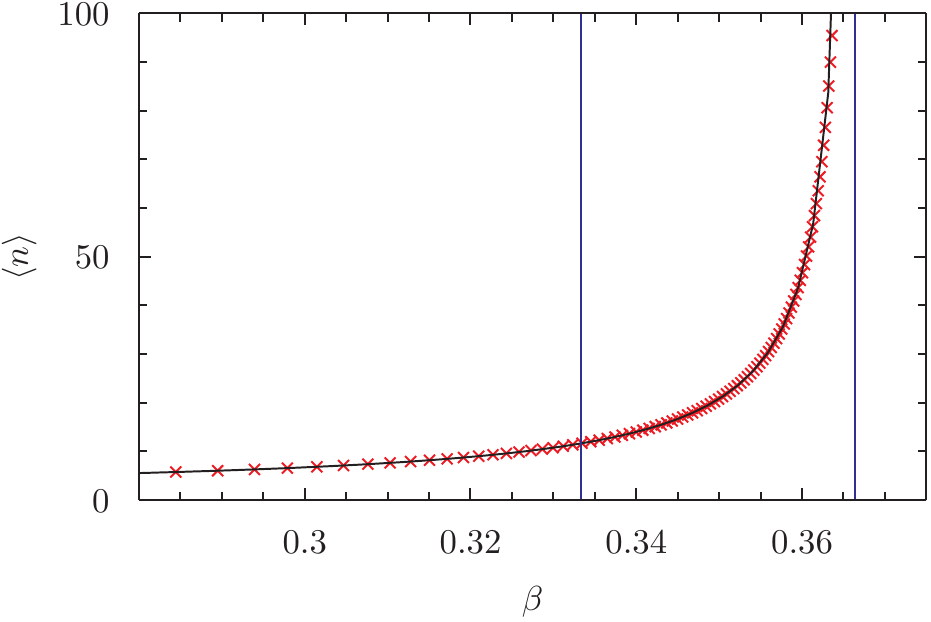}}\\
  \subfloat[$\langle a,b,c|a^2,b^2,c^2 \rangle$ sampled with $\alpha =1$.]{
    \includegraphics[height=5cm]{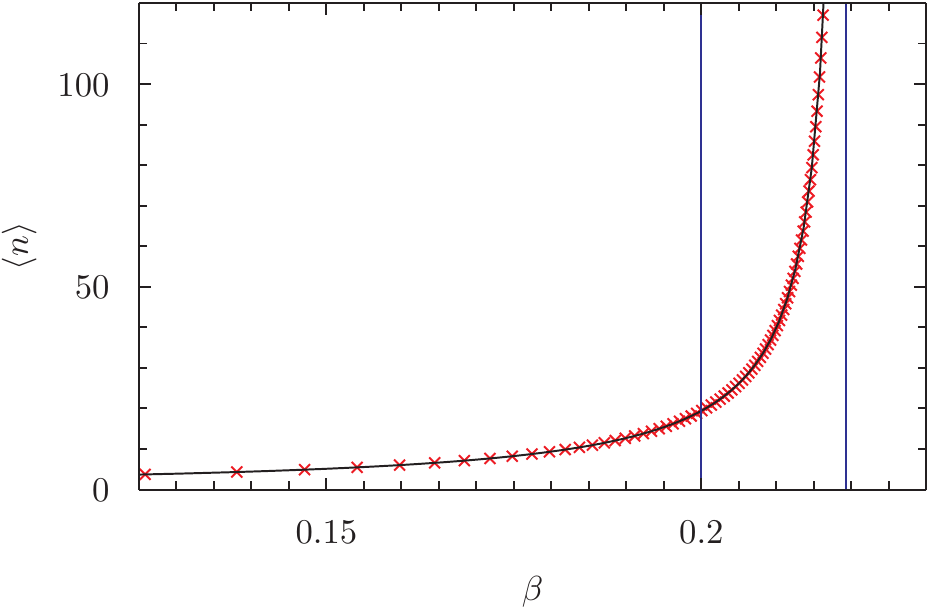}}
  \caption{Mean length of sampled words plotted against $\beta$ for $K_1, 
K_2$ and  $K_3$. The crosses indicate data obtained from the algorithm, while 
the curves indicate the expectation derived from the exact cogrowth series 
for each group. The first vertical lines in each plot indicates 
$\beta=\nicefrac{1}{3}, \nicefrac{1}{3}, \nicefrac{1}{5}$ (respectively)  and
also the reciprocal of the cogrowth where the statistic will diverge --- being
$0.3418821478, 0.3664068598$ and $0.2192752634$ respectively.}
  \label{fig:kuksov plots}
\end{figure}

In Figure~\ref{fig:kuksov plots} we compare data obtained from our algorithm 
with the exact expectation, which was computed by combining the exact cogrowth 
series above with equation~\Ref{eqn exp alpha}. Note that because the chain 
avoids the empty word, we modify the above generating functions by 
subtracting 1 from each (being the contribution from the empty word). As was 
the case for $\mathbb{Z}^2$, there is excellent agreement between the numerical 
and exact results.

\subsection{Application to $\mathrm{BS}(N,N)$}
The cogrowth series for $\mathrm{BS}(N,N)$ is not known in closed form 
for $N\geq 2$. In recent work \cite{\BScogrowth} the authors and Tom Wong demonstrate that 
the cogrowth series for $\mathrm{BS}(N,N) = \langle a,b \mid a^Nb a^{-N} 
\bbar \rangle$ is D-finite, that is, the series $C(z)$ satisfies a linear 
differential equation with polynomial coefficients. This work allows the 
cogrowth to be computed exactly for moderate values of $N$ 
in polynomial time. 

It follows that the cogrowth series can be computed to (essentially) any 
desired number of terms. Using that truncated series and equation~\Ref{eqn exp 
alpha} we the compute expectation of the length to any desired precision. In 
Figure~\ref{fig:bs2233 plots} we display the expected mean length against data 
obtained from the Markov chain. As with previous examples,  we see excellent agreement.

\begin{figure}[h!]
  \centering
  \subfloat[$\mathrm{BS}(2,2)$ sampled with $\alpha = 1$.]{
    \includegraphics[height=5cm]{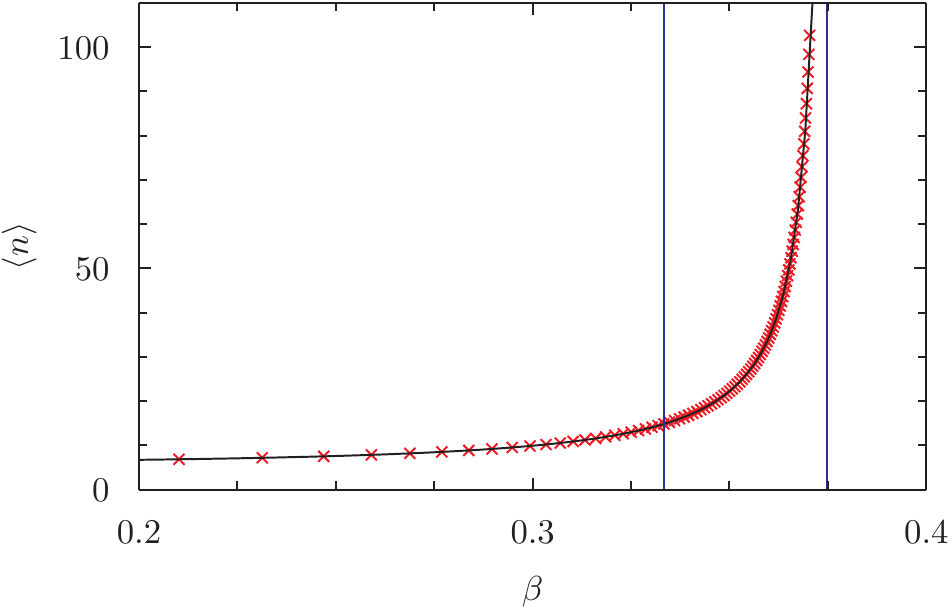}}\\
  \subfloat[$\mathrm{BS}(3,3)$ sampled with $\alpha = 1$.]{
    \includegraphics[height=5cm]{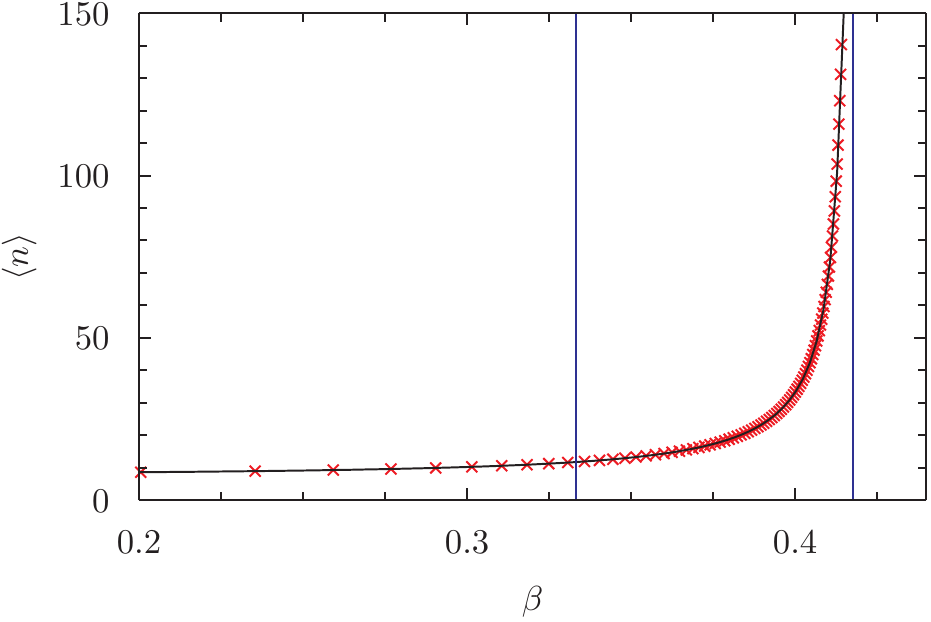}}
  \caption{Mean length of sampled words plotted against $\beta$ 
for $\mathrm{BS}(2,2)$ and $\mathrm{BS}(3,3)$. The crosses indicate data 
obtained from the algorithm, while the curves indicates the expectation derived 
from the cogrowth series for each group. The vertical lines indicate 
$\beta=\nicefrac{1}{3}$ and also the reciprocal of the cogrowth being 
$0.3747331572$ and $0.417525628$ respectively (see \cite{\BScogrowth}). We see 
excellent agreement between our numerical data and the exact results.}
  \label{fig:bs2233 plots}
\end{figure}

\subsection{Application to $\mathrm{BS}(N,M)$ with $N \neq M$.}
The work \cite{\BScogrowth} is mostly concerned with $\mathrm{BS}(N,N)$, but the
central enumerative result (Proposition~3.6 in \cite{\BScogrowth}) also holds 
for $\mathrm{BS}(N,M) = \langle a,b \mid a^Nb a^{-M} \bbar \rangle$. The 
authors 
derive a system of three $q$-algebraic 
equations which can be iterated to compute the first few terms of $C(z)$. This 
is more efficient than a brute-force approach but it still requires exponential 
time and memory.

Explicitly the authors define a two-variable generating function
\begin{align}
  G(z;q) &= \sum_n z^n g_n(q) & \text{where } g_n(q) &= \sum_k g_{n,k} q^k
\end{align}
where $g_{n,k}$ is the total number of words (not just those that are freely 
reduced) of length $n$ equal to $a^k$. Thus $g_{n,0} = d(n)$ defined in 
Lemma~\ref{lemma woe}. When $N=M$, $g_n(q)$ has at most $2n+1$ non-zero terms, 
however when $N \neq M$ the number of non-zero terms is exponential in $n$.

Due to the exponential constraint, we are only able to compute the 
first few terms of cogrowth series exactly. For example we were only able to 
compute the first 60 terms of $D(z)$ (and hence $C(z)$ by equation~\Ref{eqn 
woe}) for $\mathrm{BS}(1,2)$. Using those truncated series and 
equation~\Ref{eqn exp alpha} we get a lower bound on the exact expected mean 
length --- this is the solid curve in Figures~\ref{fig:bs1213 plots} 
and~\ref{fig:bs23 plots}.

When we generated series by the above method we noticed that the polynomials 
$g_n(q)$ are dominated by the central few terms around $q^0$, while the other 
terms (being the vast majority) were negligible. This suggests an alternate 
means to estimate $D(z)$ (and so $C(z)$) --- at each iteration of the system of 
$q$-algebraic equations we discarded all but the central $2n+1$ terms of 
$g_{n,k}$. The resulting series $\tilde{G}(z;q)$ is dominated term-by-term by 
the true $G(z;q)$, but can be computed to far more terms (indeed it is 
comparable in effort to the computation for $\mathrm{BS}(N,N)$ described 
above). We have also estimated the exact expectation using this method; it 
also gives underestimates of the true expectation. The curve is plotted as 
dotted lines in Figures~\ref{fig:bs1213 plots} and~\ref{fig:bs23 plots}.

In all four plots we see good agreement between the two estimates and the data 
from the Markov chain. As $\beta$ is increased the two estimates fall below the 
Markov chain data, with the estimate from truncated series  distinctly 
lower than the estimate from approximate series. This is consistent with the 
Markov chain giving accurate estimates of the true expected length for $\beta$ 
even quite close to $\beta_c$. In the cases of 
$\mathrm{BS}(1,2)$ and $\mathrm{BS}(1,3)$ we know the reciprocal of the cogrowth is $\nicefrac{1}{3}$ 
since they are amenable, and the Markov chain data confers with this.

\begin{figure}
  \centering
  \subfloat[$\mathrm{BS}(1,2)$ sampled with $\alpha = 1$.]{
    \includegraphics[height=5cm]{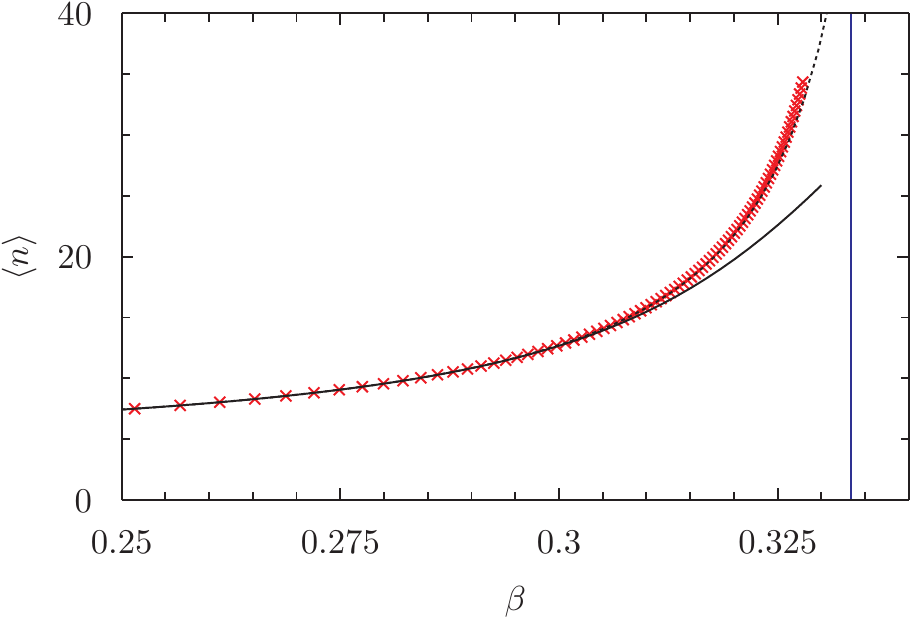}}\\
  \subfloat[$\mathrm{BS}(1,3)$ sampled with $\alpha = 2$.]{
    \includegraphics[height=5cm]{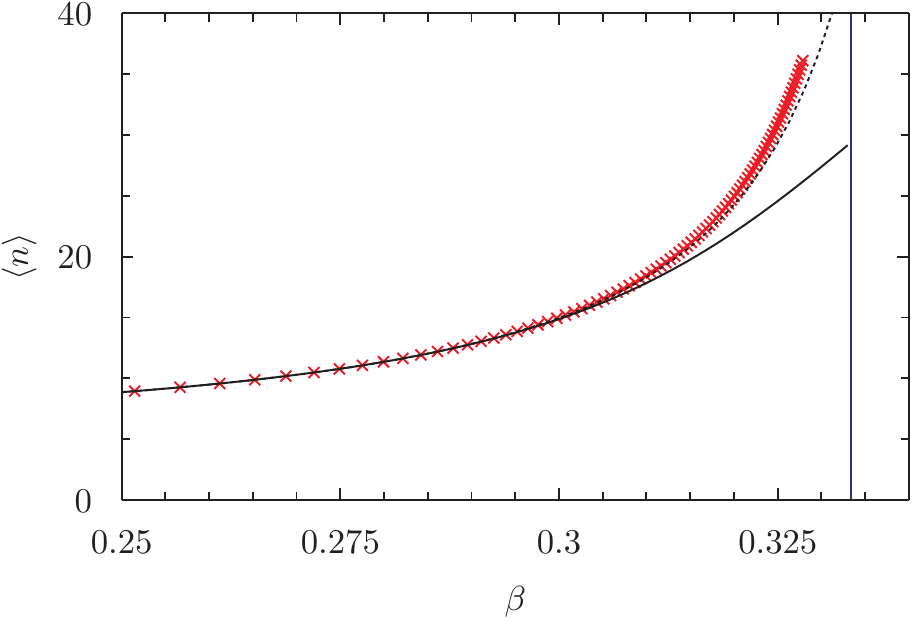}}
  \caption{Mean length of freely reduced trivial words in Baumslag-Solitar
groups $\mathrm{BS}(1,2)$ and $\mathrm{BS}(1,3)$ at different 
values of~$\beta$ and $\alpha$ as indicated. The sampled points are indicated 
with crosses, while the vertical line indicates $\beta_c=\nicefrac{1}{3}$. The 
solid line indicates estimates of the exact expectation derived from the 
exact but truncated cogrowth series. The dotted line indicates estimates of the 
expectation derived using the approximation of the cogrowth (as described in 
the 
main text). At low and moderate values of $\beta$ there is excellent agreement, 
but as $\beta$ increases the Markov chain lies above both of the 
approximations of the expectation which is consistent with the approximations 
being underestimates.}
  \label{fig:bs1213 plots}
\end{figure}

\begin{figure}
  \centering
  \subfloat[$\mathrm{BS}(2,3)$ sampled with $\alpha=1$.] {
  \includegraphics[height=5cm]{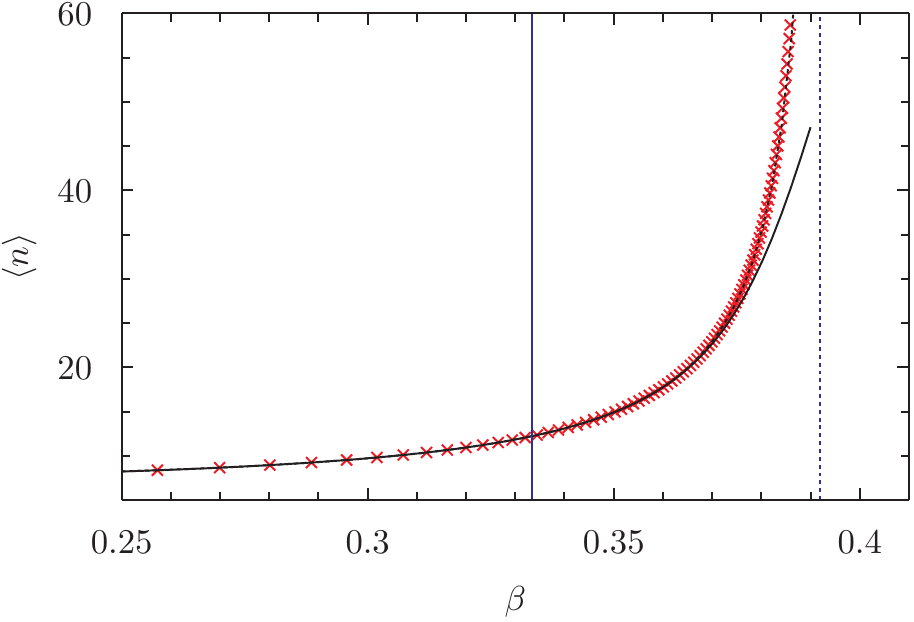} }\\
  \subfloat[$\mathrm{BS}(3,5)$ sampled with $\alpha=0$.] {
  \includegraphics[height=5cm]{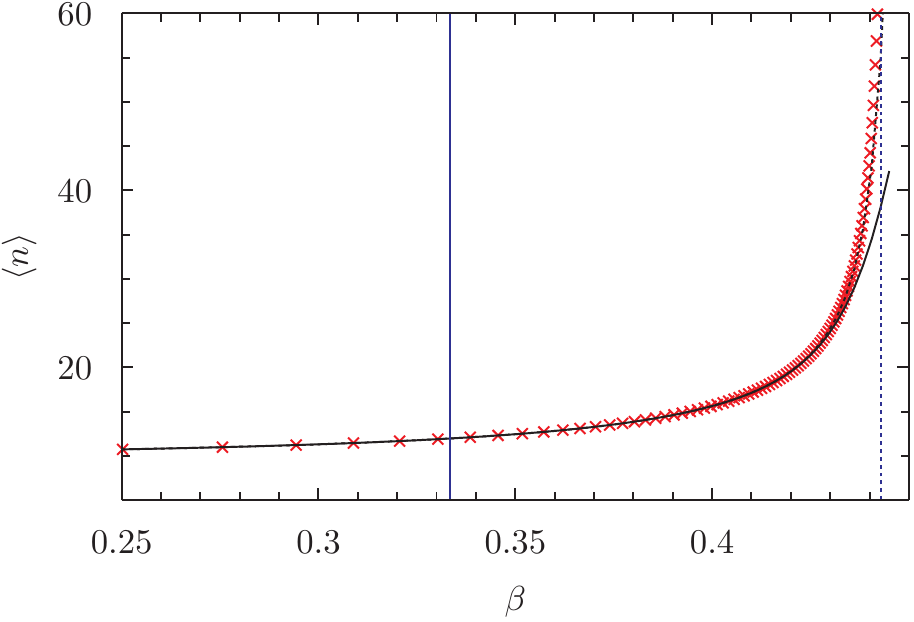} }
  \caption{The mean length of trivial words in $\mathrm{BS}(2,3)$ and 
$\mathrm{BS}(3,5)$ at different values of $\beta$. The sampled points are 
indicated with crosses, while the solid vertical line indicates 
$\beta_c=\nicefrac{1}{3}$. The dotted vertical lines indicate the estimated 
critical value of $\beta$ from analysis of the truncated series. As per the 
previous figure, the solid line indicates estimates of the expectation from 
truncated series while the dotted line indicates estimates from the approximate 
series (see the main text). At low and moderate values of $\beta$ there 
is excellent agreement, but as $\beta$ increases the Markov chain lies above 
both of the approximations of the expectation.}
  \label{fig:bs23 plots}
\end{figure}

\subsection{Application to the basilica group}
We now turn to the first of two groups for which we know very little about 
the cogrowth series --- namely the basilica group first studied by Grigorchuk and  Zuk  \cite{Grigorchuk02}. This 
group has an infinite presentation
\begin{align}
G &= \left\langle a,b \,\middle|\,   \left[ a^n,[a^n,b^n]\right] 
  \text{ and }
  \left[ b^n,[b^n,a^{2n}] \right] \text{ where $n$ is a power of $2$}
\right\rangle
\end{align}
where we have used the notation $[x,y] = x^{-1}y^{-1}xy$ and $x^y = 
y^{-1}xy$. This group embeds in the finitely presented group 
\cite{Grigorchuk02} 
\begin{equation}\label{eqn:basilicatilde}
\W{G} = \left\langle a,b \,\middle|\, a^{b^2}=a^2,\,\left[ \left[ [a,b^{-1}] 
,a\right],a\right] = 1 \right\rangle .
\end{equation}
Bartholdi and Virag proved  that  both  $G$ and $\W{G}$ are amenable \cite{\BartV}, and 
separate the classes of amenable and subexponentially amenable groups.

As noted in subsection \ref{subsec inf}
 our algorithm can be extended to infinite 
presentations, however for this article we restricted our study to the finitely presented group $\W{G}$. We ran the algorithm on three presentations derived from the above 
presentation by simple \emph{Tietze transformations} (see \cite{\LS} p. 89).  
The first is obtained from the above by putting $c=[a,b^{-1}]$, and the second 
by putting $c=a^b$. Simplification gives the representations
\begin{align}
  \W{G} &=  
    \left\langle a,b,c \,\middle|\,
    c=[a,b^{-1}],  a^{b^2}=a^2, \left[ [c,a],a \right]=1 
    \right\rangle , \label{eqnA4.4} \\
  \W{G} &=  
    \left\langle a,b,c \,\middle|\, 
c=a^b, c^b=a^2, c^{-1}aca^{-1}c^{-1}a^{-1}ca=1 
    \right\rangle . 
\label{eqnA4.5}
\end{align}
We implemented the Markov chain for both of these presentations. We 
plot the mean length of words sampled from the chains in 
Figure~\ref{fig:figurebasilica}. An immediate observation is that the mean 
length is remarkably insensitive to changes in $\beta$. Because of this we 
found that our data was far harder to analyse than for the other groups 
discussed above. This is compounded by the absense of cogrowth series data for 
comparison. 

Because this data appeared so insensitive to $\beta$, we also examined a 
measure 
of the statistical error in our estimates. To do this we consider samples from 
the Markov chain as a time series of length $N$. We slice this sequence into 
$M$ non-overlapping blocks of length $\nicefrac{N}{M}$. Let the mean length 
observed in the $i^\mathrm{th}$ such block be denoted $\langle n \rangle_i$. 
The variance in these mean lengths and our error estimate are then given by
\begin{align}
 \mathrm{var} &= \frac{1}{M^2} \cdot \sum_i \langle n \rangle_i^2 
 - \left( \frac{1}{M} \sum_i \langle n \rangle_i \right)^2  \\
 \mathrm{err} &= \sqrt{\frac{\mathrm{var}}{M-1}}
\end{align}
Our typical runs consisted of around $10^3$ blocks each of length 
approximately $10^7$. We made estimates of autocorrelations at the highest 
values of $\beta$ and found them to be much shorter than the block length. This 
validates the above estimate of the error.

We repeated this analysis on  the examples studied above (the Baumslag-Solitar groups and the examples of 
Kouksov), and found that the error estimates were very small. Indeed, if we were to place 
error-bars on our plots of the mean length they would be smaller than the 
crosses used to denote the data --- except very close to $\beta_c$. This is 
consistent with our observation that our Markov chain data agrees closely with 
exact results. It also indicates another method of detecting the location of a 
singularity --- we expect that the error estimate will diverge as $\beta \to 
\beta_c$.

\begin{figure}[ht!]
  \centering
  \subfloat[Mean length with $\alpha=5$.]{
   \includegraphics[height=5cm]{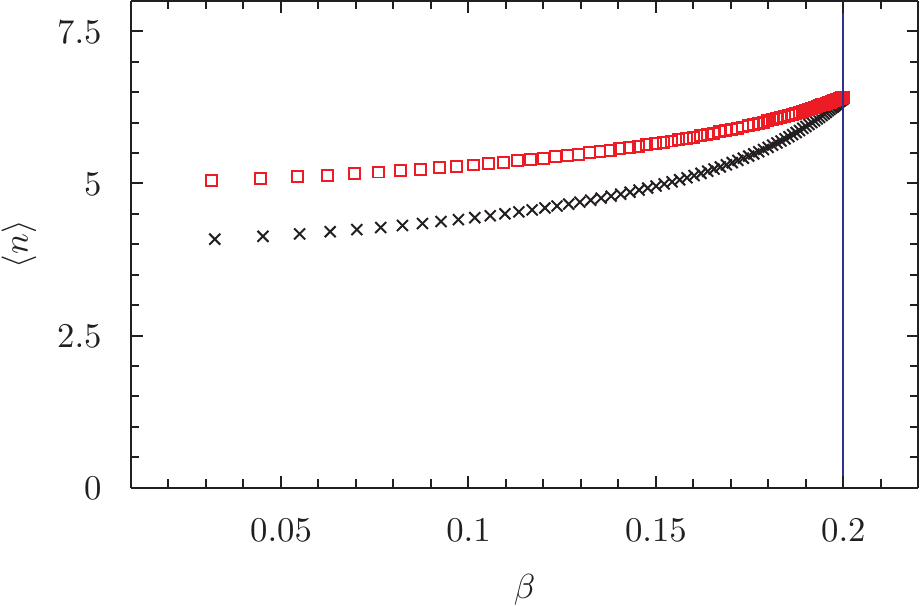}}\\
 \subfloat[$\hbox{err}^{-1}$ with $\alpha=5$.]{
   \includegraphics[height=5cm]{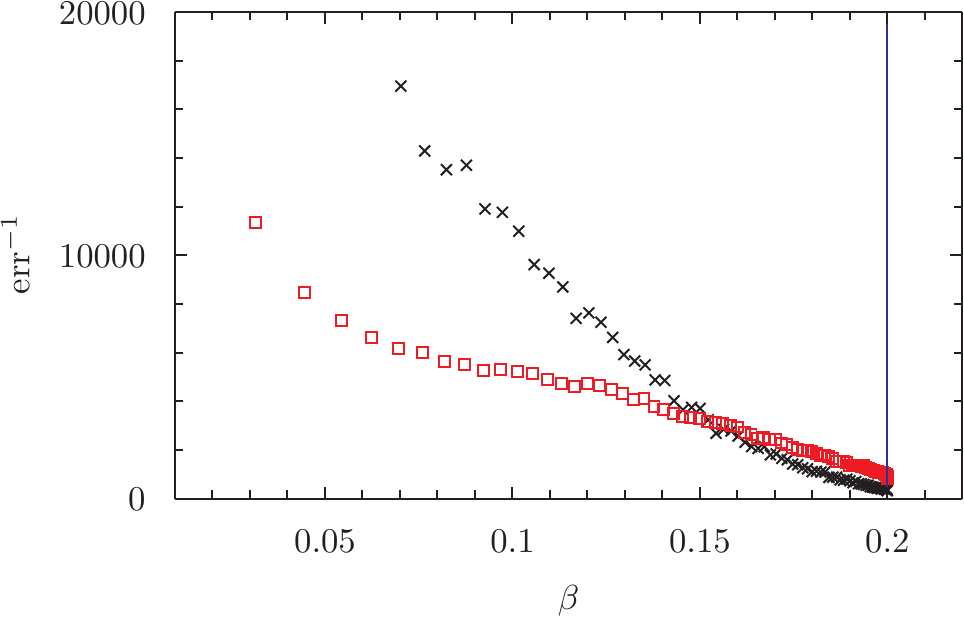}}
 \caption{(A) The mean length of words plotted against $\beta$ for two 
presentations of group $\W{G}$. Data points indicated by $\square, \times$ 
corresponds to~\Ref{eqnA4.4} and~\Ref{eqnA4.5} respectively.  (B) The 
reciprocal of the estimated error against $\beta$. Notice that as $\beta \to 
\nicefrac{1}{5}$ the error begins to diverge.}
 \label{fig:figurebasilica}
\end{figure}

We have plotted the reciprocal of our error estimate against $\beta$ for these
two presentations in Figure~\ref{fig:figurebasilica}.  We see a much clearer 
signal of divergence closer to $\beta_c=\nicefrac{1}{5}$ than we do for the 
mean length data. 

We studied a third presentation, in which the relators are of shorter and 
comparable lengths. We set  $c=a^b, d=[a,b^{-1}], e=[d,a]$ in equation~\ref{eqn:basilicatilde} to obtain the 
presentation:
\begin{align}
  \W{G} &=  
    \left\langle a,b,c,d,e \,\middle|\, c=a^b, d=[a,b^{-1}], e=[d,a], c^b=a^2, 
[e,a]=1
    \right\rangle . 
\label{eqnA4.6}
\end{align}
We found that the mean-length data from this presentation was much better 
behaved and gave a clearer signal of a singularity at $\beta = 
\nicefrac{1}{9}$. See Figure~\ref{fig basilica3}. We also analysed the error 
data and estimate that the reciprocal of the error goes to zero as $\beta \to 
1.115\pm0.005$. The data from this presentation is consistent with the 
amenability of $\W{G}$. Overall, the data from all three presentations is 
consistent with the group being amenable.

\begin{figure}
  \centering
  \subfloat[Mean length with $\alpha=1$.]{
   \includegraphics[height=5cm]{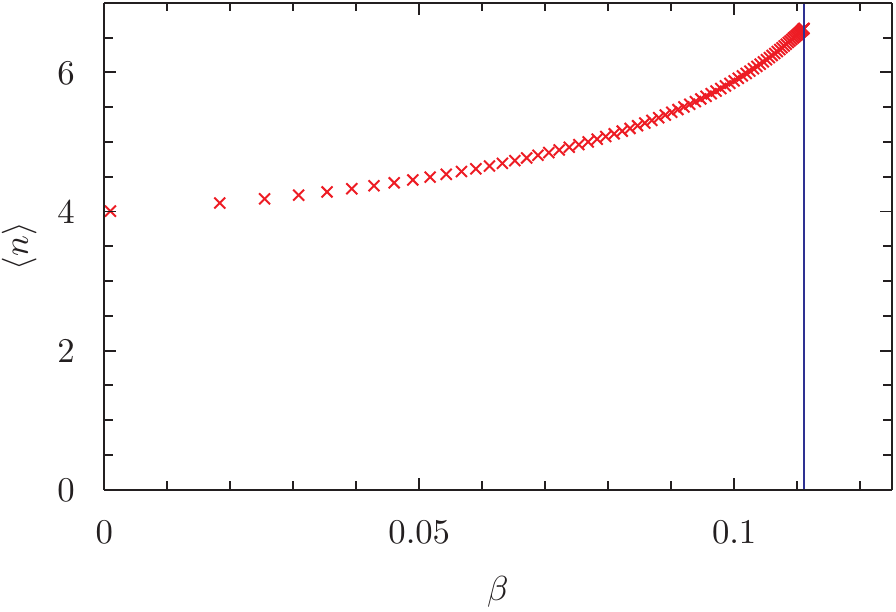}}\\
 \subfloat[$\hbox{err}^{-1}$ with $\alpha=1$.]{
   \includegraphics[height=5cm]{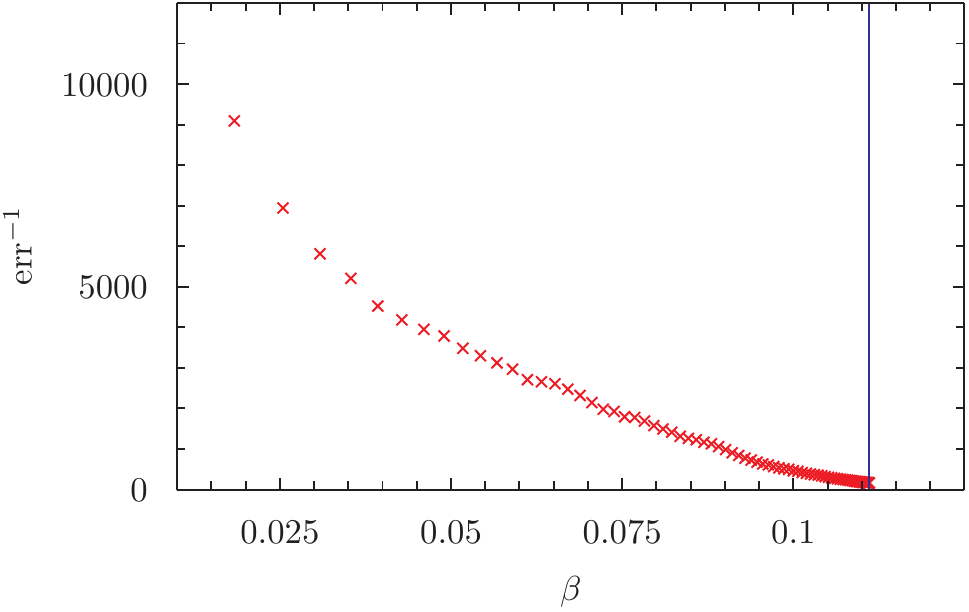}}
 \caption{(A) The mean length of words plotted against $\beta$ for the 
presentation of group $\W{G}$ in equation~\Ref{eqnA4.6}.  (B) The 
reciprocal of the estimated error against $\beta$. Notice that as $\beta \to 
\nicefrac{1}{9}$ the error begins to diverge.}
 \label{fig basilica3}
\end{figure}

\subsection{Application to the Thompson's group $F$}
We now turn to our last application, Thompson's group $F$. 
 We 
started by examining its standard finite presentation
\begin{align}
F = \big\langle a,b &\,\big|\, [a\bbar,\abar b a], [a\bbar,a^{-2} b a^2] 
\big\rangle. \label{eqnF1}
\end{align}

In addition to this  presentation, we implemented the chain on two 
further presentations derived using 
simple Tietze transformations:
\begin{align}
F=  \big\langle a,b,c,d &\,\big|\, c=\abar b a, d=\abar c a, [a\bbar,c], 
[a\bbar, d]
\big\rangle, \label{eqnF2}\\
F=  \big\langle a,b,c,d,e  &\,\big|\, c=\abar b a, d=\abar c a, e=a\bbar, 
[e,c], [e,d] \big\rangle. \label{eqnF3}
\end{align}
Note that the generators $a,b,c,d$ above are usually denoted
$x_0,x_1,x_2,x_3$ respectively in the Thompson's group literature.

We display the mean length computed from our Markov chain for these three 
presentations in Figure~\ref{fig:thomp plots}. 
In all cases we also saw no 
indication of a singularity at the amenable values of $\beta=\nicefrac{1}{3}$, $\beta=\nicefrac{1}{7}$ 
and $\beta=\nicefrac{1}{9}$ respectively. We also repeated the error-analysis that was done 
for $\W{G}$ above --- see Figure~\ref{fig:thomp err plots}. Again we saw no 
indication of a singularity present in these statistics at the amenable value 
of 
$\beta$. We have made rough estimates of the location of the dominant 
singularity of the cogrowth series by estimating where the reciprocal of the 
observed error goes to zero.  The data from these presentations were easier to 
analyse than that from $\W{G}$ and because of this we were able to obtain 
estimates with tighter error bars. Our analysis gives
\begin{align}
 \beta_c &= 0.395\pm0.005, 0.172\pm 0.002 \mbox{ and } 0.134\pm0.004
\end{align}
for the three presentations. These imply cogrowths of approximately 
$2.53\pm0.03, 5.81\pm0.07$ and $7.4\pm0.2$, all of which are well below the 
amenable values of 3,7 and 9.

Of course, these 
estimates do not constitute a proof that Thompson's 
group is non-amenable. However, they are stronger numerical evidence than 
any previous work (such as \cite{MR2473819, MR2395786} and \cite{MR3043436}). 
As 
is the case with almost any numerical experiment, one cannot rule out the 
presence of particular pathalogical behaviours in Thompson's group that distort 
the behaviour of the chain and so the numerical data.  

\begin{figure}
  \centering
  \subfloat[Standard presentation (\ref{eqnF1}) for $F$  sampled with $\alpha = 
2$]{
    \includegraphics[height=5cm]{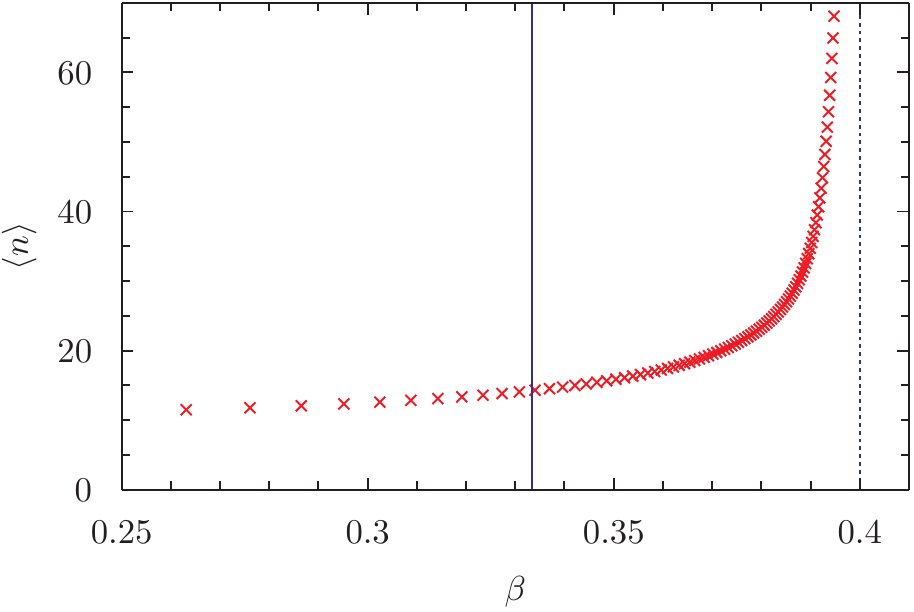}}\\
  \subfloat[Presentation (\ref{eqnF2})   for $F$  sampled with $\alpha = 2$]{
    \includegraphics[height=5cm]{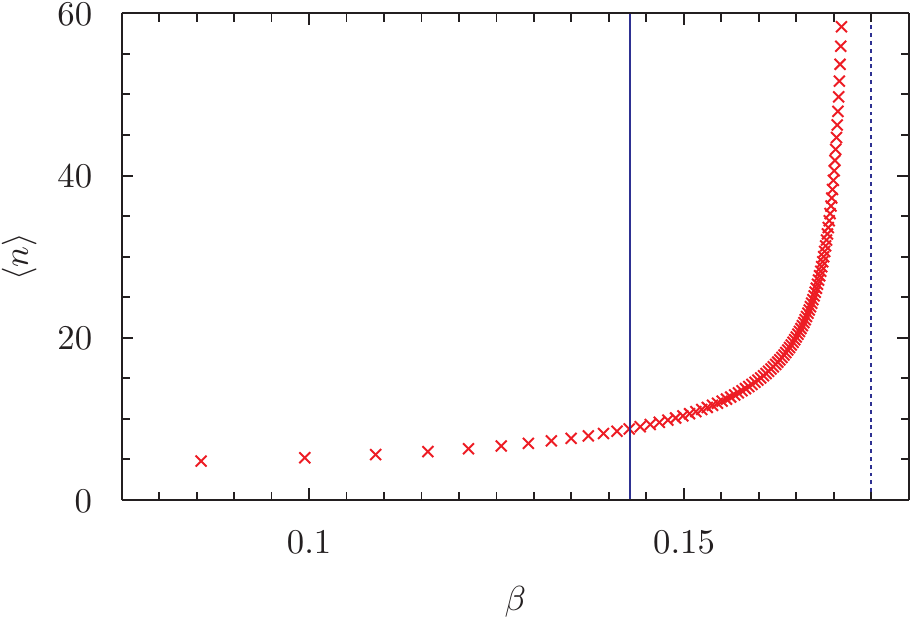}}\\
  \subfloat[Presentation (\ref{eqnF3})   for $F$ sampled with $\alpha = 1$]{
    \includegraphics[height=5cm]{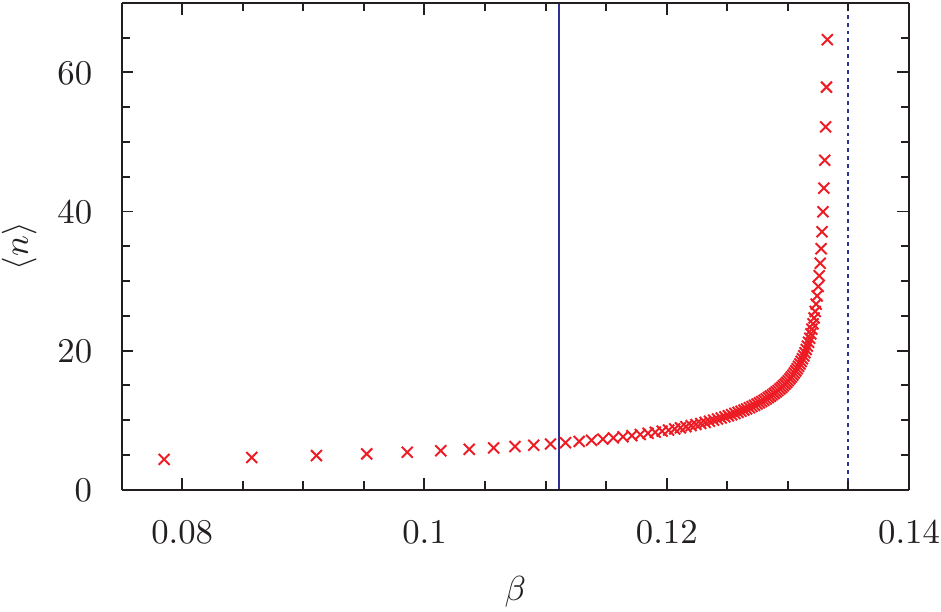}}
  \caption{Mean length of freely reduced trivial words in Thompson's group $F$
at different values of~$\beta$. The solid blue lines indicate the reciprocal of
the cogrowth of amenable groups with $k$ generators~$\beta_c =
\nicefrac{1}{(2k-1)}$. The dashed blue lines indicate the approximate
location of the vertical asymptote. In each case, we see that the mean length
of trivial words is finite for $\beta$-values  slightly above~$\beta_c$.}
  \label{fig:thomp plots}
\end{figure}

\begin{figure}
  \centering
  \subfloat[Standard presentation  (\ref{eqnF1})  for $F$  sampled with $\alpha 
= 0,1,2,3$.]{
    \includegraphics[height=5cm]{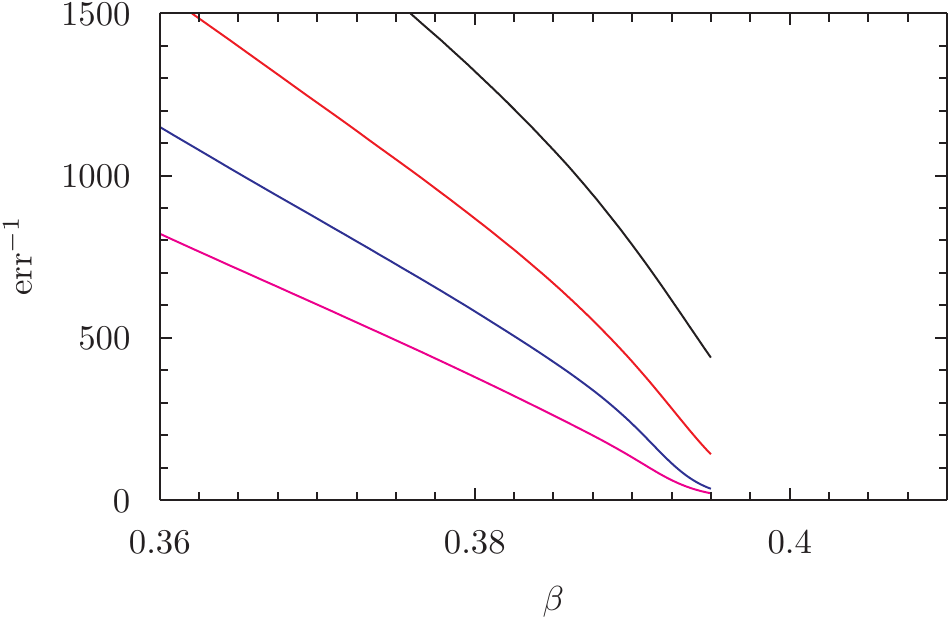}}\\
  \subfloat[Presentation  (\ref{eqnF2})    for $F$ sampled with $\alpha = 
0,1,2,3$.]{
    \includegraphics[height=5cm]{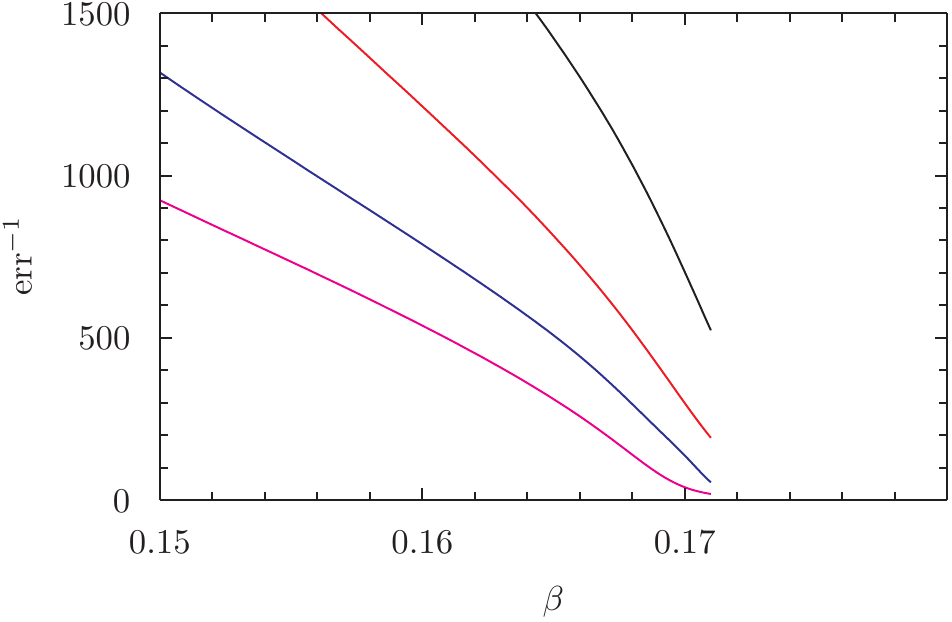}}\\
  \subfloat[Presentation  (\ref{eqnF3})   for $F$ sampled with $\alpha = 
0,1,2,3$.]{
    \includegraphics[height=5cm]{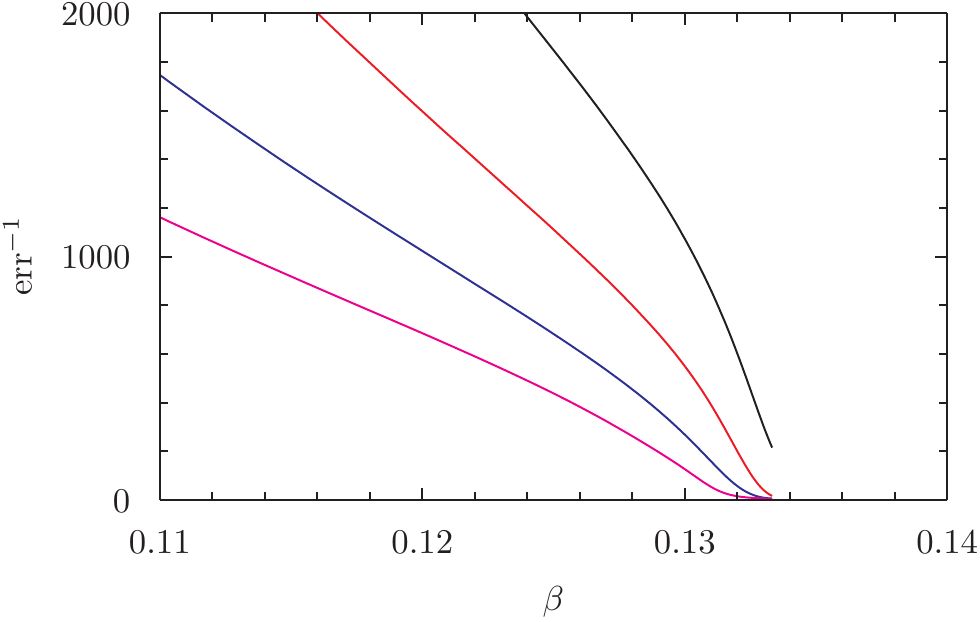}}
  \caption{The reciprocal of the estimated standard error of the mean length as
a function of $\beta$ for the three presentations of Thompson's group. In each
plot we show 4 curves corresponding to simulations at $\alpha=0,1,2,3$
(anti-clockwise from top). Extrapolating these curves leads to estimates of
$\beta_c$ of $0.395\pm0.005$, $0.172\pm 0.002$, $0.134\pm0.004$. These are all
well above the values of amenable groups.}
  \label{fig:thomp err plots}
\end{figure}

\section{Conclusions}
\label{sec conc}
We have introduced a novel Markov chain which samples trivial words from 
finitely presented groups. Since this chain operates on the state space of 
trivial words rather than on the Cayley graph, it is quite different from 
previous studies of random walks on groups. We have shown that the Markov chain 
converges to the stationary distribution $\pi$ and so asymptotically samples 
from it. Further, $\pi$ is a stretched Boltzmann distribution related to the 
cogrowth series of the presentation and so statistics collected from the chain 
inform us about the cogrowth of the group.

We have implemented the chain for presentations of both amenable and 
non-amenable groups for which the cogrowth series is known exactly. In these 
cases we observe excellent agreement between statistics collected from our 
chain and exact results. We have also implemented the chain for presentations 
of groups for which little is known about the cogrowth series. In the case of 
the basilica group (or more precisely a finitely presented group into which 
the basilica group embeds), our results are consistent with the amenability of 
the group. On the other hand, our results for Thompson's group $F$ suggest that 
it is not amenable. 

In cases where the cogrowths series is known exactly (or can be computed 
to arbitrary precision) we observed that the mean length statistic generated by our 
chains converged quickly to the correct value. This behaviour held for both 
amenable and non-amenable groups.

As is the case with any numerical experiment we cannot rule out the presence of 
pathologies influencing our results. This raises two obvious questions which lie 
beyond this present work: how can we determine the rate at which the Markov 
chain convergences to the stationary distribution; and how can we analyse 
statistics from the chain to obtain precise estimates of the asymptotic 
behaviour of the cogrowth function. Both of these questions have strong implications for 
numerical tests of the amenability of a group, and we intend to pursue them in future work.

\section*{Acknowledgements}
The authors thank Sean Cleary, Tony Guttmann and Stu Whittington for
helpful discussions about this work. Much of the numerical work was run on the 
Westgrid computer cluster and the authors thank Westgrid for their support.
This research was supported by the Australian Research Council (ARC), the the 
Natural Sciences and Engineering Research Council of Canada (NSERC), and 
Perimeter Institute for Theoretical Physics.  Research at Perimeter Institute is 
supported by the Government of Canada through Industry Canada and by the 
Province of Ontario through the Ministry of Economic Development and Innovation.

\bibliographystyle{plain}
\bibliography{refs-rand}

\def\cprime{$'$}
\begin{thebibliography}{10}

\bibitem{MR2473819}
G.~N. Arzhantseva, V.~S. Guba, M.~Lustig, and J.~Pr{\'e}aux.
\newblock Testing {C}ayley graph densities.
\newblock {\em Ann. Math. Blaise Pascal}, 15(2):233--286, 2008.

\bibitem{MR2730578}
L.~Bartholdi, V.~A. Kaimanovich, and V.~V. Nekrashevych.
\newblock On amenability of automata groups.
\newblock {\em Duke Math. J.}, 154(3):575--598, 2010.

\bibitem{MR2176547}
L.~Bartholdi and B.~Vir{\'a}g.
\newblock Amenability via random walks.
\newblock {\em Duke Math. J.}, 130(1):39--56, 2005.

\bibitem{MR2131635}
L.~Bartholdi and W.~Woess.
\newblock Spectral computations on lamplighter groups and {D}iestel-{L}eader
  graphs.
\newblock {\em J. Fourier Anal. Appl.}, 11(2):175--202, 2005.

\bibitem{MR2197808}
J.~M. Belk and K.~S. Brown.
\newblock Forest diagrams for elements of {T}hompson's group {$F$}.
\newblock {\em Internat. J. Algebra Comput.}, 15(5-6):815--850, 2005.

\bibitem{MR2395786}
J.~Burillo, S.~Cleary, and B.~Wiest.
\newblock Computational explorations in {T}hompson's group {$F$}.
\newblock In {\em Geometric group theory}, Trends Math., pages 21--35.
  Birkh\"auser, Basel, 2007.

\bibitem{Cohen}
J.~M. Cohen.
\newblock Cogrowth and amenability of discrete groups.
\newblock {\em J. Funct. Anal.}, 48(3):301--309, 1982.

\bibitem{MR1245303}
P.~Diaconis and L.~Saloff-Coste.
\newblock Comparison techniques for random walk on finite groups.
\newblock {\em Ann. Probab.}, 21(4):2131--2156, 1993.

\bibitem{MR1233621}
P.~Diaconis and L.~Saloff-Coste.
\newblock Comparison theorems for reversible {M}arkov chains.
\newblock {\em Ann. Appl. Probab.}, 3(3):696--730, 1993.

\bibitem{MR1254308}
P.~Diaconis and L.~Saloff-Coste.
\newblock Moderate growth and random walk on finite groups.
\newblock {\em Geom. Funct. Anal.}, 4(1):1--36, 1994.

\bibitem{MR1414925}
P.~Diaconis and L.~Saloff-Coste.
\newblock Random walks on finite groups: a survey of analytic techniques.
\newblock In {\em Probability measures on groups and related structures, {XI}
  ({O}berwolfach, 1994)}, pages 44--75. World Sci. Publ., River Edge, NJ, 1995.

\bibitem{MR1650316}
P.~Diaconis and L.~Saloff-Coste.
\newblock Walks on generating sets of groups.
\newblock {\em Invent. Math.}, 134(2):251--299, 1998.

\bibitem{MR2216708}
K.~Dykema.
\newblock Symmetric random walks on certain amalgamated free product groups.
\newblock In {\em Topological and asymptotic aspects of group theory}, volume
  394 of {\em Contemp. Math.}, pages 87--99. Amer. Math. Soc., Providence, RI,
  2006.

\bibitem{DykemaBS}
K.~Dykema and D.~Redelmeier.
\newblock Lower bounds for the spectral radii of adjacency operators on
  {B}aumslag-{S}olitar groups.
\newblock {\em Preprint, arXiv:1006.0556}, 2010.

\bibitem{BScogrowth}
M.~Elder, A.~Rechnitzer, E.~J. Janse~van Rensburg, and T.~Wong.
\newblock The cogrowth series for $\mathrm{BS}(n,n)$ is {D}-finite.
\newblock {\em Preprint, arXiv:1309.4184}, 2013.

\bibitem{MR3043436}
M.~Elder, A.~Rechnitzer, and T.~Wong.
\newblock On the cogrowth of {T}hompson's group {$F$}.
\newblock {\em Groups Complex. Cryptol.}, 4(2):301--320, 2012.

\bibitem{geyer}
C.~J. Geyer and E.~A. Thompson.
\newblock Annealing {M}arkov chain {M}onte {C}arlo with applications to
  ancestral inference.
\newblock {\em Journal of the American Statistical Association}, pages
  909--920, 1995.

\bibitem{Grig}
R.~I. Grigorchuk.
\newblock Symmetrical random walks on discrete groups.
\newblock In {\em Multicomponent random systems}, volume~6 of {\em Adv. Probab.
  Related Topics}, pages 285--325. Dekker, New York, 1980.

\bibitem{Grigorchuk02}
R.~I. Grigorchuk and \.{Z}uk A.
\newblock On a torsion-free weakly branch group defined by a three state
  automaton.
\newblock In {\em International conference on geometric and combinatorial
  methods in group theory and semigroup theory}, volume~12 of {\em Internat. J.
  Algebra Comput.}, pages 223--246. World Scientific, Singapore, 2002.

\bibitem{janse2009monte}
E.~J. Janse~van Rensburg.
\newblock Monte {C}arlo methods for the self-avoiding walk.
\newblock {\em Journal of Physics A: Mathematical and Theoretical}, 42:323001,
  2009.

\bibitem{karlin1975}
S.~Karlin and H.E. Taylor.
\newblock {\em A first course in stochastic processes}.
\newblock Elsevier, 1975.

\bibitem{MR1487319}
D.~Kouksov.
\newblock On rationality of the cogrowth series.
\newblock {\em Proc. Amer. Math. Soc.}, 126(10):2845--2847, 1998.

\bibitem{MR1689726}
D.~Kouksov.
\newblock Cogrowth series of free products of finite and free groups.
\newblock {\em Glasg. Math. J.}, 41(1):19--31, 1999.

\bibitem{MR2275700}
S.~P. Lalley.
\newblock The weak/strong survival transition on trees and nonamenable graphs.
\newblock In {\em International {C}ongress of {M}athematicians. {V}ol. {III}},
  pages 637--647. Eur. Math. Soc., Z\"urich, 2006.

\bibitem{MR0577064}
R.~C. Lyndon and P.~E. Schupp.
\newblock {\em Combinatorial group theory}.
\newblock Springer-Verlag, Berlin, 1977.
\newblock Ergebnisse der Mathematik und ihrer Grenzgebiete, Band 89.

\bibitem{MR1197356}
N.~Madras and G.~Slade.
\newblock {\em The self-avoiding walk}.
\newblock Probability and its Applications. Birkh\"auser Boston Inc., Boston,
  MA, 1993.

\bibitem{MR2894945}
A.~Mann.
\newblock {\em How groups grow}, volume 395 of {\em London Mathematical Society
  Lecture Note Series}.
\newblock Cambridge University Press, Cambridge, 2012.

\bibitem{metropolis}
N.~{Metropolis}, A.~W. {Rosenbluth}, M.~N. {Rosenbluth}, A.~H. {Teller}, and
  E.~{Teller}.
\newblock {Equation of State Calculations by Fast Computing Machines}.
\newblock {\em Journal of Chemical Physics}, 21:1087--1092, June 1953.

\bibitem{mitzenmacher2005}
M.~Mitzenmacher and E.~Upfal.
\newblock {\em Probability and computing: Randomized algorithms and
  probabilistic analysis}.
\newblock Cambridge University Press, 2005.

\bibitem{MR3095713}
J.~T. Moore.
\newblock Fast growth in the {F}\o lner function for {T}hompson's group {$F$}.
\newblock {\em Groups Geom. Dyn.}, 7(3):633--651, 2013.

\bibitem{MR1691645}
T.~Nagnibeda.
\newblock An upper bound for the spectral radius of a random walk on surface
  groups.
\newblock {\em Zap. Nauchn. Sem. S.-Peterburg. Otdel. Mat. Inst. Steklov.
  (POMI)}, 240(Teor. Predst. Din. Sist. Komb. i Algoritm. Metody. 2):154--165,
  293--294, 1997.

\bibitem{MR2087798}
T.~Nagnibeda.
\newblock Random walks, spectral radii, and {R}amanujan graphs.
\newblock In {\em Random walks and geometry}, pages 487--500. Walter de Gruyter
  GmbH \& Co. KG, Berlin, 2004.

\bibitem{MR2338235}
R.~Ortner and W.~Woess.
\newblock Non-backtracking random walks and cogrowth of graphs.
\newblock {\em Canad. J. Math.}, 59(4):828--844, 2007.

\bibitem{rosenthal2006}
J.S. Rosenthal.
\newblock {\em A first look at rigorous probability theory}.
\newblock World Scientific, 2006.

\bibitem{tesimonte}
M.~C. Tesi, E.~J. Janse~van Rensburg, E.~Orlandini, and S.~G. Whittington.
\newblock Monte {C}arlo study of the interacting self-avoiding walk model in
  three dimensions.
\newblock {\em Journal of Statistical Physics}, 82(1):155--181, 1996.

\bibitem{MR1251963}
S.~Wagon.
\newblock {\em The {B}anach-{T}arski paradox}.
\newblock Cambridge University Press, Cambridge, 1993.
\newblock With a foreword by Jan Mycielski, Corrected reprint of the 1985
  original.

\bibitem{MR731608}
W.~Woess.
\newblock Cogrowth of groups and simple random walks.
\newblock {\em Arch. Math. (Basel)}, 41(4):363--370, 1983.

\end{thebibliography}

\end{document}